\documentclass[reqno]{amsart}
\usepackage{amsfonts}
\usepackage{amssymb}
\usepackage{color}

\theoremstyle{plain}
\newtheorem{thm}{Theorem}

\newtheorem{lem}{Lemma}
\newtheorem{proposition}{Proposition}

\theoremstyle{definition}
\newtheorem{defn}{Definition}
\theoremstyle{remark}
\newtheorem{rem}{Remark}
\numberwithin{equation}{section}

\makeatletter
\newcommand{\rmnum}[1]{\romannumeral #1}
\newcommand{\Rmnum}[1]{\expandafter\@slowromancap\romannumeral #1@}
\makeatother

\begin{document}
\title{Hausdorff dimensions and quasisymmetric minimalities of some homogeneous Moran sets}
\author{Jun Li}
\address[Jun Li]{School of Mathematics, Guangxi University, Nanning, 530004, P.~R. China}
\email{lijun1999star@163.com}

\author{Yanzhe Li$^{*}$}
\address[Yanzhe Li]{School of Mathematics, Guangxi University; Guangxi Center for Mathematical Research; Center for Applied Mathematics of Guangxi(Guangxi University), Nanning, 530004, P.~R. China}
\email{liyz@gxu.edu.cn}

\author{Pingping Liu}
\address[Pingping Liu]{School  of Mathematics, Guangxi University, Nanning, 530004, P.~R. China}
\email{lpplrx@163.com}

\thanks{This work is supported by National Natural Science Foundation of China (No.12461015) and Guangxi Natural Science Foundation (2020GXNSFAA297040). }
\thanks{*Corresponding author. }
\thanks{These authors contributed equally to this work.}

\begin{abstract}
In this paper, we study the Hausdorff dimensions and the quasisymmetric  minimalities of some homogeneous Moran sets. We obtain a Hausdorff dimension formula for three classes of homogeneous Moran sets which satisfy some conditions. We also show that two classes of them with Hausdorff dimension 1 are quasisymmetrically  Hausdorff-minimal.
\end{abstract}

\maketitle

\section{Introduction}
\label{sec:intro}

The Hausdorff dimensions of the fractals sets is a hot research topic in the study of fractal geometry.
There are many important results about the Hausdorff dimensions of the homogeneous Moran sets. Feng, Wen and Wu\cite{fdj97} studied the Hausdorff dimensions of the homogeneous Moran sets and obtained the range of values of all homogeneous Moran sets, they also showed that the Hausdorff dimensions of some homogeneous Moran sets can reach the maximum or minimum value.
Wen and Wu\cite{wzywuj05} defined the homogeneous perfect sets by giving some restrictions on the gaps between the basic intervals of the homogeneous Moran sets, and  showed  the Hausdorff dimension formula of the homogeneous perfect sets  under some conditions.

In this paper, we obtain a Hausdorff dimension formula  of three classes homogeneous Moran sets, which generalizes the result in \cite{wzywuj05}.

 Let $(X,d_X)$ and $(Y,d_Y)$ be two metric spaces, and $f$ be a homeomorphism between $X$ and $Y$. We call $f$ a quasisymmetric mapping  if there is a homeomorphism $\eta:[0,\infty)\to [0,\infty)$, such that for all triples $a, b, x$ of  distinct points in $X$,
\begin{equation*}
  \frac{d_Y(f(x),f(a))}{d_Y(f(x),f(b))}\le\eta(\frac{d_X(x,a)}{d_X(x,b)}).
\end{equation*}
If $X$ and $Y$ are both $\mathbb{R}^{n}$, we say that $f$ is a $n$-dimensional quasisymmetric mapping.

 The quasisymmetric mappings contain the bi-Lipschitz mappings, however, some properties of them are quite different. The bi-Lipschitz mappings preserve the fractal dimensions, but the fractal dimensions of the fractal sets can be changed under some quasisymmetric mappings. We call a set $E\subset\mathbb{R}^{n}$ quasisymmetrically Hausdorff-minimal if $\dim_{H}f(E)\ge \dim_{H}E$ for all $n-$dimensional quasisymmetric mapping $f$, where $\dim_{H}E$ is denoted by the Hausdorff dimension of $E$.

The quasisymmetrically minimalities for Hausdorff dimensions of the sets have received a substantial amount of attentions in recent years.
 It is not difficult to prove that any set $E\subset \mathbb{R}^n$ with $\dim_{H}E=0$ is quasisymmetrically Hausdorff-minimal\cite{Ahl06}.
Kovalev\cite{klv06} and Bishop\cite{bcj99} obtained that if $E\subset \mathbb{R}$ and satisfies $0<\dim_{H}E<1$, then $E$ is not a quasisymmetrically Hausdorff-minimal set. Gehring and Vaisala\cite{gfw,gfw2} found that when $n\ge2$, any set $E\subset \mathbb{R}^n$ with $\dim_{H}E=n$ is quasisymmetrically Hausdorff-minimal. However, Tukia\cite{tp89} pointed out that a set $E\subset \mathbb{R}$ with $\dim_{H}E=1$ may not be quasisymmetrically Hausdorff-minimal.
So, there is a question: which sets in $\mathbb{R}$ with Hausdorff dimension 1 are quasisymmetrically Hausdorff-minimal?
Staples and Ward\cite{ssg98} obtained that all quasisymmetrically thick sets are  quasisymmetrically Hausdorff-minimal.
Hakobyan\cite{hh06} showed that the middle interval Cantor sets with Hausdorff dimension 1 are quasisymmetrically Hausdorff-minimal.
Hu and Wen\cite{hmd08} generalized the result of \cite{hh06} to the uniform Cantor sets with Hausdorff dimension 1 under the condition that the sequence $\{n_k\}$ is bounded.
Wang and Wen\cite{ww14} generalized the result of \cite{hmd08} without assuming the boundedness of $\{n_k\}$.
Dai et al.\cite{dyx11} obtained a large class of Moran sets with Hausdorff dimension 1 is quasisymmetrically Hausdorff-minimal.
Yang, Wu and Li\cite{yjj18},  Xiao and Zhang{\cite{XYQ}} showed that the homogeneous perfect sets with Hausdorff dimension 1 are quasisymmetrically Hausdorff-minimal under some conditions, which generalized the result of \cite{ww14}.

In this paper, we  prove that two classes of homogeneous Moran sets with Hausdorff dimension 1 are quasisymmetrically Hausdorff-minimal, which generalizes the results in  \cite{yjj18}  and \cite{XYQ}.

\bigskip

\section{Preliminaries}
\label{sec:pre}
\subsection{Homogeneous Moran Sets}
We recall the definition of the homogeneous Moran sets.

Let  $\left\{c_{k}\right\}_{k\ge1}$ be a sequence of positive real numbers and $\left\{n_{k}\right\}_{k\ge1}$ be a sequence of positive integers such that $n_k\ge 2$ and $n_kc_k<1$ for any $k\ge 1$.  For any $k\ge1$, let $D_{k}=\left\{i_{1}i_{2}\cdots i_{k}:1\le i_{j} \le n_{j},1\le j \le k\right\}$, $D_{0}=\emptyset$ and $D=\cup_{k\ge0}D_{k}$. If $\sigma=\sigma_{1}\sigma_{2}\cdots\sigma_{k}\in D_{k}$, $\tau=\tau_{1}\tau_{2}\cdots\tau_{m}(1\le \tau_j\le n_{k+j}, 1\le j \le m)$, then $\sigma*\tau =\sigma_{1}\sigma_{2}\cdots\sigma_{k}\tau_{1}\tau_{2}\cdots\tau_{m}\in D_{k+m}$.

\begin{defn}\label{HMS}(Homogeneous Moran sets \rm{\cite{hua00}})
Suppose that $I_0=[0,1]$ and $\mathcal I=\{I_\sigma:\sigma\in D \}$ is a collection of the closed subintervals of $I_0$. We call $I_0$ the initial interval. We say that the collection $\mathcal I$ satisfies the homogeneous Moran structure provided:
\begin{enumerate}
\item[\textup{(1)}] If $\sigma=\emptyset$, we have $I_{\sigma}=I_{0}$;
\item[\textup{(2)}] For any $k\ge1$ and $\sigma\in{D_{k-1}}$, $I_{\sigma * 1}, \cdots, I_{\sigma * n_{k}}$ are closed subintervals of $I_{\sigma}$  with $\min(I_{\sigma*(l+1)})\ge \max(I_{\sigma*l})$ for any $1\le l \le n_{k}-1$, which means the interiors of $I_{\sigma*l}$ and $I_{\sigma*(l+1)}$ are disjoint;
\item[\textup{(3)}] For any $k\ge1$ and $\sigma\in{D_{k-1}}$, $1\le i \le j \le n_{k}$, we have $$\frac{\left|I_{\sigma * i}\right|}{\left|I_{\sigma}\right|}=\frac{\left|I_{\sigma * j}\right|}{\left|I_{\sigma}\right|}=c_{k},$$
where $\left|A\right|$ denotes the diameter of the set $A(A\subset \mathbb{R})$. We call $c_k$ the $k$-order contracting ratio.
\end{enumerate}

If $\mathcal{I}$ satisfies the homogeneous Moran structure, let $E_{k}=\cup_{\sigma\in{D_{k}}}I_{\sigma}$ for any $k\ge0$,  then the  nonempty compact set
	$E=\cap_{k\ge0}E_{k}=E(I_{0},\{n_{k}\},\{c_{k}\})$ is called a homogeneous Moran set.
	For any $k\ge0$, let $\mathcal{I}_{k}=\left\{I_{\sigma}:\sigma\in{D_{k}}\right\}$, then any $I_{\sigma}\in \mathcal{I}_{k}$ is called a $k$-order basic interval of $E$. We use $\mathcal{M}(I_{0},\{n_{k}\},\{c_{k}\})$ to denote the class of all homogeneous Moran sets associated with $I_{0},\{n_{k}\},\{c_{k}\}$.

\end{defn}

\bigskip
We give some marks for further discussions. For any $k\ge 1$ and $\sigma \in D_{k-1}$,
$1\le i \le n_k-1$, let
\begin{equation}
 \begin{aligned}
     \min(I_{\sigma*1})-\min(I_{\sigma})&=\eta_{\sigma,0};\\
     \min(I_{\sigma*(i+1)})-\max(I_{\sigma*i})&=\eta_{\sigma,i};\\
     \max(I_{\sigma})-\max(I_{\sigma*n_k})&=\eta_{\sigma,n_k},\nonumber
 \end{aligned}
 \end{equation}
then for any $k\ge 1$, $\{\eta_{\sigma,l}:\sigma \in D_{k-1}, 0\le l \le n_k\}$ is a sequence of nonnegative real numbers. For any  $\sigma \in D_{k-1}, 1\le l \le n_k-1$, we call $\eta_{\sigma,l}$ the length of a $k$-order gap of $E$.

For any $k\ge 1$, let $\bar\alpha_k$ be the maximum value of the length of a $k$-order gap  of $E$ and $\underline\alpha_k$ be the minimum value of the length of a $k$-order gap of $E$, which means
$$\bar\alpha_k=\max\limits_{\sigma\in D_{k-1}, 1\leq j\leq n_k-1}\eta_{\sigma,j},\ \underline\alpha_k=\min\limits_{\sigma\in D_{k-1}, 1\leq j\leq n_k-1}\eta_{\sigma,j}.$$
Let $N_k$ be the number of the $k$-order basic intervals of $E$ and $\delta_k$ be the length of any $k$-order basic interval of $E$, then
\begin{align}
\nonumber
N_k=\prod_{i=1}^{k} n_i, \delta_k=\prod_{i=1}^{k} c_i.
\end{align}
Let $l(E_k)$ be the total length of all $k$-order basic intervals of $E$, then $l(E_k)=N_k\delta_k$.

\begin{rem}
If $k> 1$, then  the cardinality of $\{\eta_{\sigma,l}:\sigma \in D_{k-1}, 0\le l \le n_k\}$ is $N_{k-1}(n_{k}+1)$. If $k=1$,  the cardinality of $\{\eta_{\sigma,l}:\sigma \in D_{k-1}, 0\le l \le n_k\}$ is $n_1+1$.
\end{rem}

\begin{rem}\label{rm} For any  $k\ge 1$, $\sigma_1 \in D_{k-1}$, $\sigma_2 \in D_{k-1}$, $\sigma_1\ne \sigma_2$  and $0\le l \le n_k$, , $\eta_{\sigma_1,l}$ may not be equal to $\eta_{\sigma_2,l}$. But if $E=E(I,\{n_{k}\}, \{c_{k}\},\{\eta_{k,j}\})$ is a homogeneous perfect set(the definition can be seem in \cite{wzywuj05}), then for any  $k\ge 1$, $\sigma_1 \in D_{k-1}$, $\sigma_2 \in D_{k-1}$, $\sigma_1\ne \sigma_2$  and $0\le l \le n_k$, $\eta_{\sigma_1,l}=\eta_{\sigma_2,l}=\eta_{k,l}$.
\end{rem}

\begin{rem}
More results about the fractal dimensions and the quasisymmetric minimalities of the homogeneous Moran sets can be found in \cite{lss24,LW10,LW11,CWW17,LFY21,DDW23,LLY25}.

\end{rem}
\bigskip

\subsection{Some Lemmas}
The following lemmas will play important roles in the proof of the theorems of this paper.

\begin{lem}\label{l1}\textmd{\rm (Mass distribution principle \rm{\cite{wen00}},\rm{\cite{fkj97}})}
Suppose that $s\ge0$, let $\mu$ be a mass distribution  on a Borel set  $E\subseteq\mathbb{R}${\rm(}which means $\mu$ is a positive and finite Borel measure on $E${\rm)}\begin{itemize}
  \item[(i)]If there are two positive constants $c_1$ and $\eta_1$, such that
  $\mu(U)\le c_1\left | U \right |^{s}$ for any set $U\subseteq\mathbb{R}$ with $0\le\left | U \right | \le \eta_1$,
  then $\dim_{H}E\ge s$;
  \item[(ii)]If there are two positive constants $c_2$ and $\eta_2$, such that $\mu(B(x,r))\le c_2r^s$,
  for all $x\in E$ and $0\le r\le\eta_2$, then $\dim_{H}E\ge s$.

\end{itemize}

It is noteworthy that $\mathrm{(i)}$ and $\mathrm{(ii)}$ are equivalent.
\end{lem}
The mass distribution principle is a useful tool to estimate the lower bound of the Hausdorff dimensions of the sets.
\bigskip

For any interval $I\subseteq\mathbb{R}$ and $\rho>0$, let $\rho I$ be the interval which has the same center with $I$ and length of it is $\rho \left|I\right|$. We have the following lemma.

\begin{lem}\label{l2}{\rm{(\cite{wjm93})}}
Let $f:\mathbb{R}\to \mathbb{R}$ be a \text{\rm1}-dimensional quasisymmetric mapping, then
 for any two intervals $I'$, $I $ with $I'\subseteq I$,
 there exist positive real numbers $\lambda>0$, $K_{\rho}>0$ and $0<p\le 1\le q$ such that
	$$\lambda(\frac{|I^{'}|}{\left|I\right|})^{q}\le \frac{|f(I^{'})|}{\left|f(I)\right|}\le 4(\frac{|I^{'}|}{\left|I\right|})^{p},\quad
\frac{\left|f(\rho I)\right|}{\left|f(I)\right|}\le K_{\rho}.$$
\end{lem}
Lemma \ref{l2} gives the  relationship  between the lengths for the image sets  and the lengths for the original sets of the quasissymmertic mappings.

\bigskip

\section{Main results}
Theorem \ref{thm} and Theorem \ref{thm2} are our main results.

\begin{thm}\label{thm}
    Let
	$E\in \mathcal{M}(I_{0},\left\{n_{k}\right\},\left\{c_{k}\right\})$ be a homogeneous Moran set which satisfies the following condition:
there exist two sequences of nonnegative real numbers $\{L_k\}_{k\ge1}$ and $\{R_k\}_{k\ge 1}$, such that
\begin{equation}
 \begin{aligned}
    \eta_{\sigma,0}=L_{k+1},\quad
    \eta_{\sigma,n_{k+1}}=R_{k+1} \nonumber
 \end{aligned}
 \end{equation}
 for any $k\ge 0$, $\sigma\in D_{k}$.

And if for any $k\ge 1$, at least one of the following three conditions is satisfied:
\begin{itemize}
  \item[(A)]there exists $\omega_1>0$, such that $\bar\alpha_k\le \omega_1 \underline\alpha_k$;
  \item[(B)]there exists $\omega_2>0$, such that $\bar\alpha_k\le \omega_2\cdot c_1c_2\cdots c_k$;
  \item[(C)]there exists $\omega_3>0$, such that $n_k\underline\alpha_k\ge \omega_3\cdot c_1c_2\cdots c_{k-1}$.
\end{itemize}

Then
\begin{equation}\label{311}
\dim_HE=\liminf\limits_{k \to\infty}  \frac{\log n_1n_2\cdots n_k}{-\log(\delta_k-L_{k+1}-R_{k+1})}.
\end{equation}
\end{thm}

\begin{rem}\label{rm1}
 If $E=E(I_{0},\{n_{k}\}, \{c_{k}\},\{\eta_{k,j}\})$ is a homogeneous perfect set(the definition can be seem in \cite{wzywuj05}),  then $E\in \mathcal{M}(I_{0},\left\{n_{k}\right\},\left\{c_{k}\right\})$  is a homogeneous Moran set with $\eta_{\sigma,0}=\eta_{k,0}=L_{k}$, $\eta_{\sigma,n_k}=\eta_{k,n_k}=R_{k}$, and $\eta_{\sigma,l}=\eta_{k,l}$ for any $k\ge 1$, $\sigma\in D_{k-1}$ and $1\le l\le n_k-1$, if  $E$ satisfies the condition (A) or (B) or (C) in Theorem 1.2 of \cite{wzywuj05}, it is obvious that $E$ satisfies the condition (A) or (B) or (C) in Theorem \ref{thm} of this paper. On the other hand, if $E\in \mathcal{M}(I_{0},\left\{n_{k}\right\},\left\{c_{k}\right\})$  is a homogeneous Moran set which satisfies the conditions of Theorem \ref{thm} of this paper, $E$ may not be a homogeneous perfect set by Remark \ref{rm} of this paper. Notice that equation (1) of Theorem 1.2 of \cite{wzywuj05} is equal to equation (\ref{311}) of this paper, thus Theorem \ref{thm} of this paper generalizes Theorem 1.2 of \cite{wzywuj05}.
\end{rem}

\bigskip

\begin{thm}\label{thm2}
	    Let
	$E\in \mathcal{M}(I_{0},\left\{n_{k}\right\},\left\{c_{k}\right\})$ be a homogeneous Moran set which satisfies the following condition:
there exist two sequences of nonnegative real numbers $\{L_k\}_{k\ge1}$ and $\{R_k\}_{k\ge 1}$, such that
\begin{equation}
 \begin{aligned}
    \eta_{\sigma,0}=L_{k+1},\quad
    \eta_{\sigma,n_{k+1}}=R_{k+1} \nonumber
 \end{aligned}
 \end{equation}
 for any $k\ge 0$, $\sigma\in D_{k}$.

If $\dim_{H}E=1$, and for any $k\ge 1$, at least one of the following two conditions is satisfied:
\begin{itemize}
  \item[(A)]there exists $\omega_1>0$, such that $\bar\alpha_k\le \omega_1 \underline\alpha_k$;
  \item[(B)]there exists $\omega_2>0$, such that $\bar\alpha_k\le \omega_2\cdot c_1c_2\cdots c_k$.
\end{itemize}

Then we have $\operatorname{dim}_{H}f(E)=1$ for any \text{\rm1}-dimensional quasisymmetric mapping $f$, which implies that $E$ is a quasisymmetrically  Hausdorff-minimal set.

\end{thm}

\begin{rem}
Similar to the analysis in Remark \ref{rm1}, we can obtain that if $E=E(I_{0},\{n_{k}\}, \\\{c_{k}\},\{\eta_{k,j}\})$ is a homogeneous perfect set which satisfies the conditions in Theorem 1 of \cite{yjj18} or Theorem 2.2 of \cite{XYQ}, then $E\in \mathcal{M}(I_{0},\left\{n_{k}\right\},\left\{c_{k}\right\})$ is a homogeneous Moran set which satisfies the conditions of Theorem \ref{thm} of this paper. On the other hand, if $E\in \mathcal{M}(I_{0},\left\{n_{k}\right\},\left\{c_{k}\right\})$  is a homogeneous Moran set which satisfies the conditions of Theorem \ref{thm} of this paper, $E$ may not be a homogeneous perfect set by Remark \ref{rm} of this paper. Thus Theorem \ref{thm2} of this paper generalizes Theorem 1 of \cite{yjj18} and Theorem 2.2 of \cite{XYQ}.

\end{rem}
\bigskip

\section{The first reconstruction of the Homogeneous Moran sets}

For the convenience of further discussions, we reconstruct the homogeneous Moran set $E\in \mathcal{M}(I_{0},\left\{n_{k}\right\},\left\{c_{k}\right\})$  which satisfies the conditions of Theorem \ref{thm} and Theorem \ref{thm2}.

For any $k\ge 0$, $\sigma\in D_{k}$, let $I_{\sigma}^{*}$ be a closed subinterval of $I_{\sigma}$ satisfying the following conditions:
\begin{enumerate}
  \item[(a)]$\min(I^{*}_\sigma)-\min(I_\sigma)=\eta_{\sigma,0}=L_{k+1},\quad
            \max(I_\sigma)-\max(I^{*}_\sigma)=\eta_{\sigma,n_{k+1}}=R_{k+1}$;
  \item[(b)]$\left | I^{*}_ {\sigma}\right | =\sum_{l=1}^{n_{k+1}-1} \eta_{\sigma,l}+n_{k+1}c_1c_2\cdots c_{k+1}=\delta_k-L_{k+1}-R_{k+1}$.
\end{enumerate}

 Let $I_{0}^{*}=I_{\emptyset}^{*}$, denote $\delta^*_{0}=\left|I_{0}^{*}\right|$, $\delta^*_{k}=\left|I_{\sigma}^{*}\right|$ for any $ k\ge 1$ and $\sigma\in D_k$. Write $E_{k}^{*}=\cup_{\sigma\in{D_{k}}}I_{\sigma}^{*}$ for any $k\ge 0$, then
\begin{equation}
E=\bigcap_{k\ge 0}\bigcup_{\sigma \in D_k}I_{\sigma}=\bigcap_{k\ge 0}\bigcup_{\sigma \in D_k}I^{*}_{\sigma}=\bigcap_{k\ge 0}E_{k}^{*}.
\end{equation}
We call $I^*_\sigma$ a $k$-order first reconstructed basic interval of $E$ for any $ k\ge 1$ and $\sigma\in D_k$.

In fact, $E\in \mathcal{M}(I_{0}^{*},\left\{n_{k}^{*}\right\},\left\{c_{k}^{*}\right\})$ is a homogeneous Moran set with the following parameters for any $k\ge 0$, and $\sigma \in D_k$:
\begin{enumerate}	\item[\textup{(1)}]$I_{0}^{*}=I_{0}-[\min(I_{0}),\min(I_{0})+\eta_0)-(\max(I_{0})-\eta_{n_1},\max(I_{0})]$;
	\item[\textup{(2)}]$c_{k+1}^{*}=\frac{\delta^*_{k+1}}{\delta^*_{k}}$, $n_{k+1}^{*}=n_{k+1}$.
\end{enumerate}

 For any $k\ge 1$ and $\sigma \in D_{k-1}$,
$1\le i \le n_k-1$, let
\begin{equation}
 \begin{aligned}
     \min(I^{*}_{\sigma*1})-\min(I^{*}_{\sigma})&=\eta^{*}_{\sigma,0};\\
     \min(I^{*}_{\sigma*(i+1)})-\max(I^{*}_{\sigma*i})&=\eta^{*}_{\sigma,i};\\
     \max(I^{*}_{\sigma})-\max(I^{*}_{\sigma*n_k})&=\eta^{*}_{\sigma,n_k}.\nonumber
 \end{aligned}
 \end{equation}
Then for any $k\ge 1$, $\{\eta^{*}_{\sigma,l}:\sigma \in D_{k-1}, 0\le l \le n_k\}$ is a sequence of nonnegative real numbers,  and for any  $\sigma \in D_{k-1}, 1\le l \le n_k-1$, we call $\eta^{*}_{\sigma,l}$ the length of a $k$-order  first reconstructed gap of $E$.

For any $k\ge 0$, $\sigma\in D_{k}$ and $1\le l \le n_{k+1}-1$, we have
\begin{align}\label{aa}
\eta^*_{\sigma,l}=\eta_{\sigma,l}+\eta_{\sigma*l,n_{k+2}}+\eta_{\sigma*(l+1),0}=\eta_{\sigma,l}+R_{k+2}+L_{k+2}
  ,\end{align}\begin{align}\label{ab}
 \eta^*_{\sigma,0}=\eta_{\sigma*1,0}=L_{k+2}, \quad
 \eta^*_{\sigma,n_{k+1}}=\eta_{\sigma*{n_{k+1}},n_{k+2}}=R_{k+2}.
\end{align}
Define $ L^*_{k+1}=\eta_{\sigma*1,0}=L_{k+2}, \quad
   R^*_{k+1}=\eta_{\sigma*{n_{k+1}},n_{k+2}}=R_{k+2}.$

For any $k\ge 1$, let $N^*_{k}$ be the number of the $k$-order first reconstructed basic intervals of $E$ and $\delta^*_{k}$ be the length of any $k$-order first reconstructed basic interval of $E$, we have
 $$N^*_{k}=n^*_1n^*_2\cdots n^*_{k},\quad \delta^*_{k}=\delta^*_0c^*_1c^*_2\cdots c^*_{k}.$$

For any $k\ge0, \sigma \in D_{k}$, let $e_{k+1}=\sum_{l=1}^{n_{k+1}-1}\eta_{\sigma,l}, e^{*}_{k+1}=\sum_{l=1}^{n_{k+1}-1}\eta^{*}_{\sigma,l}$,
then by (\ref{aa}),
\begin{equation}
e^*_{k+1}=\sum_{l=1}^{n_{k+1}-1}(\eta_{\sigma,l}+R_{k+2}+L_{k+2})=e_{k+1}+(n_{k+1}-1)(R_{k+2}+L_{k+2}). \nonumber
\end{equation}

For any $k\ge 0$, let $\bar\alpha^*_{k+1}$ be the maximum value of the length of a $(k+1)$-order  first reconstructed  gap of $E$ and $\underline\alpha^*_{k+1}$ be the minimum value of  the length of a $(k+1)$-order  first reconstructed  gap  of $E$, which means
$$\bar\alpha^*_{k+1}=\max\limits_{\sigma\in D_{k}, 1\leq j\leq n_{k+1}-1}\eta^*_{\sigma,j},
\underline\alpha^*_{k+1}=\min\limits_{\sigma\in D_{k}, 1\leq j\leq n_{k+1}-1}\eta^*_{\sigma,j},$$
then by (\ref{aa}) we  obtain
\begin{equation}\label{444}
\bar\alpha^*_{k+1}=\bar\alpha_{k+1}+L_{k+2}+R_{k+2};
\end{equation}
\begin{equation}\label{445}
\underline\alpha^*_{k+1}=\underline\alpha_{k+1}+L_{k+2}+R_{k+2}.
\end{equation}
Obviously,
\begin{equation}\label{42}
\underline\alpha_{k+1}\le \underline\alpha^*_{k+1},\quad  \bar\alpha_{k+1}\le\bar\alpha^*_{k+1}.
\end{equation}

Notice that for any $k\ge 0, \sigma \in D_{k}$, $\eta^*_{\sigma,0}+\eta^*_{\sigma,n_{k+1}}=\eta_{\sigma*1,0}+\eta_{\sigma*{n_{k+1}},n_{k+2}}=L_{k+2}+R_{k+2}=L^{*}_{k+1}+R^{*}_{k+1} $ and $\underline\alpha^*_{k+1}=\underline\alpha_{k+1}+L_{k+2}+R_{k+2}, $ then
\begin{equation}\label{447}
\eta^*_{\sigma,0}+\eta^*_{\sigma,n_{k+1}}=L^{*}_{k+1}+R^{*}_{k+1}\le \underline\alpha^*_{k+1} \le \bar\alpha^*_{k+1}.
\end{equation}

Since  $n^{*}_{k}=n_{k}$ and $\delta^{*}_{k}=\delta_{k}-L_{k+1}-R_{k+1}$ for any $k\ge 1$, we have

\begin{equation}
\liminf\limits_{k \to\infty}  \frac{\log n_1n_2\cdots n_k}{-\log(\delta_k-L_{k+1}-R_{k+1})}=\liminf\limits_{k \to\infty}  \frac{\log n^{*}_1n^{*}_2\cdots n^{*}_k}{-\log\delta^*_k}.\nonumber
\end{equation}
Then if we want  to get (\ref{311}), we only need to prove
\begin{equation*}
\dim_{H}E=\liminf\limits_{k \to\infty}  \frac{\log n^{*}_1n^{*}_2\cdots n^{*}_k}{-\log\delta^*_k}=\liminf\limits_{k \to\infty}  \frac{\log n_1n_2\cdots n_k}{-\log\delta^*_k}.
\end{equation*}

\bigskip

\section{The Hausdorff dimensions of the homogeneous Moran sets}
Let $E\in \mathcal{M}(I_{0},\left\{n_{k}\right\},\left\{c_{k}\right\})$  be a homogeneous Moran set which satisfies the conditions of Theorem \ref{thm} and $s=\liminf\limits_{k \to\infty}  \frac{\log n_1n_2\cdots n_k}{-\log\delta^*_k}$, to prove Theorem \ref{thm}, we need to prove $\dim_{H}E=s$. We divide the proof  into two parts.
\subsection{The estimation of   the upper bound of the Hausdorff dimension}

For any $t>s$, there exists  $\{l_k\}_{k\ge 1}\subset \mathbb{Z}^{+}$ which is strictly monotonically increasing, and there is a positive integer $K$, such that for any $k\ge K$, we have
\begin{equation}
\frac{\log n_1n_2\cdots n_{l_k}}{-\log\delta^*_{l_k}}<t,\nonumber
\end{equation}
which implies that $n_1n_2\cdots n_{l_k}{(\delta^*_{l_k})}^t<1$. Let $\mathcal{A}$ be the collection of all  $l_k$-order first reconstructed basic intervals of $E$, which means $\mathcal{A}=\{I_{\sigma}^{*}:\sigma \in D_{l_k}\}$, then $\mathcal{A}$ is a $\delta^*_{l_k}-$covering of $E$ and $\# \mathcal{A}=n_1n_2\cdots n_{l_k}$, where $\# \mathcal{A}$ denotes the cardinality of $\# \mathcal{A}$. Thus, we obtain that
\begin{equation}
\mathcal{H}^t(E)=\lim_{k \to\infty}\mathcal{H}^{t}_{\delta^*_{l_k}}(E)\le \lim_{k \to\infty} n_1n_2\cdots n_{l_k}{(\delta^*_{l_k})}^t\le 1,\nonumber
\end{equation}
which yields $\dim_HE\le t$. Since $t>s$ is arbitrary, we have $\dim_{H}E\le s$.
\subsection{The estimation of  the lower bound of the Hausdorff dimension}

It is obvious that $\dim_{H}E\ge s$ if $s=0$, without loss of generality, we assume $s>0$, then for any $0<t<s$,  there exists $k_0 \in \mathbb{Z}^{+}$ such that for any $k\ge k_0$,
\begin{equation}
\frac{\log n_1n_2\cdots n_k}{-\log\delta^*_k}>t,\nonumber
\end{equation}
which implies
\begin{equation}\label{51}
n_1n_2\cdots n_k{(\delta^*_{k})}^t>1.
\end{equation}

Let $\mu$ be a mass distribution on $E$ such that for each $k$-order first reconstructed basic interval of $E$, denoted by $I^*$, we have $\mu(I^*)=(n_1n_2\cdots n_k)^{-1}$.

Suppose that $U\subset \mathbb{R}$ is a set with $0<\left | U \right | <\delta^*_{k_0}$ and $k\ge k_0$ is an integer such that $\delta^*_{k+1}\le \left | U \right | <\delta^*_k$. Then the number of the $k$-order first reconstructed  basic intervals of $E$ which intersect $U$ is at most 2. Next we use the following 3 lemmas to estimate  $\mu(U)$.

\begin{lem}\label{l3}
If condition {\rm(A)} of Theorem {\rm\ref{thm}} holds,
which implies that there exists $\omega_1>0$ such that $\bar{\alpha}_{k}\le \omega_1\underline{\alpha}_{k}$ for any $k\ge 1$, then
$$\mu(U)\le 32\omega_1\left | U \right |^t .$$
\end{lem}

\begin{proof}
Let $k\ge k_0$ be the integer such that $\delta^*_{k+1}\le \left | U \right | <\delta^*_k$. Since $\bar{\alpha}_{k}\ge \underline{\alpha}_{k}$, we have $\omega_1\ge 1$. Then by (\ref{444}) and (\ref{445}), we have
\begin{equation}\label{52}
\bar\alpha^*_{k+1}=\bar\alpha_{k+1}+L_{k+2}+R_{k+2}\\
\le \omega_1\underline\alpha_{k+1}+L_{k+2}+R_{k+2}\le \omega _1\underline\alpha^*_{k+1}.
\end{equation}

Next, we distinguish the proof into two cases.

\textbf{Case 1:} $\delta^*_{k+1}>\underline\alpha^*_{k+1}$.
In this case, for any $\sigma \in D_k$ , since $n_{k+1}\ge 2$ and $\omega_1\ge 1$, by (\ref{aa}) and (\ref{52}), we have
\begin{equation}\label{53}
\begin{aligned}
 \delta^*_k&=\sum_{l=1}^{n_{k+1}-1} \eta_{\sigma,l}+n_{k+1}c_1c_2\cdots c_{k+1}\\
&=\sum_{l=1}^{n_{k+1}-1} \eta_{\sigma,l}+n_{k+1}(\delta^*_{k+1}+L_{k+2}+R_{k+2})\\
&\le \sum_{l=1}^{n_{k+1}-1} \eta_{\sigma,l}+2(n_{k+1}-1)(\delta^*_{k+1}+L_{k+2}+R_{k+2})\\
&\le 2\sum_{l=1}^{n_{k+1}-1}(\eta_{\sigma,l}+\delta^*_{k+1}+L_{k+2}+R_{k+2})\\
&= 2\sum_{l=1}^{n_{k+1}-1}(\eta^*_{\sigma,l}+\delta^*_{k+1})\\
&\le 2\sum_{l=1}^{n_{k+1}-1}(\omega_1\underline{\alpha}^*_{k+1} +\delta^*_{k+1})\\
&\le 4 \omega_1n_{k+1}\delta^*_{k+1}.
\end{aligned}
\end{equation}

Since the number of the $k$-order first reconstructed  basic intervals of $E$ which intersect $U$ is at most 2, the number of the $(k+1)$-order first reconstructed basic intervals of $E$ which intersect $U$ is at most $2n_{k+1}$. On the other hand, the number of the $(k+1)$-order first reconstructed basic intervals of $E$ which intersect $U$ is at most $2(\frac{\left | U \right |}{\delta^*_{k+1}}+1)\le \frac{4\left | U \right |}{\delta^*_{k+1}}$. Notice that $k\ge k_{0}$ and $\omega_1\ge 1$, hence by (\ref{51}) and (\ref{53}), we obtain that
\begin{equation}\label{54}
\begin{aligned}
 \mu(U)&\le \frac{1}{n_1n_2\cdots n_{k+1}}\min\{\frac{4\left | U \right |}{\delta^*_{k+1}},2n_{k+1}\}\\
       &\le \frac{1}{n_1n_2\cdots n_{k+1}}(\frac{4\left | U \right |}{\delta^*_{k+1}})^t(2n_{k+1})^{1-t}\\
&\le \frac{8}{n_1n_2\cdots n_k(n_{k+1}\delta^*_{k+1})^t} \left | U \right |^t\\
&\le (4\omega_1 )^t8\left | U \right |^t \frac{1}{n_1n_2\cdots n_k(\delta^*_k)^t}\\
&\le (4\omega_1 )^t8\left | U \right |^t\\
&\le 32\omega_1\left | U \right |^t.
\end{aligned}
\end{equation}

\textbf{Case 2:} $\delta^*_{k+1}\le \underline\alpha^*_{k+1}$.
In this case, by the similar proof of (\ref{53}), we have the following inequality:
\begin{equation}\label{55}
\delta^*_k\le 4\omega_1n_{k+1}\underline\alpha^*_{k+1}.
\end{equation}

\textbf{(a)} If $\left | U \right |\ge \underline\alpha^*_{k+1}$, then the number of the $(k+1)$-order first reconstructed basic intervals of $E$ which intersect $U$ is at most $2(\frac{\left | U \right |}{\underline\alpha^*_{k+1}}+1)\le \frac{4\left | U \right |}{\underline\alpha^*_{k+1}}$. By the similar proof of (\ref{54})(replace $\delta^*_{k+1}$ by $\underline\alpha^*_{k+1}$),  we have
\begin{equation}\label{56}
\mu(U)\le 32\omega_1\left | U \right |^t.
\end{equation}

\textbf{(b)} If $\left | U \right |< \underline\alpha^*_{k+1}$,  then the number of the $(k+1)$-order reconstructed basic intervals of $E$ which intersect $U$ is at most 2. Notice that $k\ge k_{0}$ and $\omega_{1}\ge 1$, then by (\ref{51}), we have
\begin{equation}\label{57}
\mu(U)\le \frac{2}{n_1n_2\cdots n_{k+1}}=\frac{2}{n_1n_2\cdots n_{k+1}(\delta^*_{k+1})^t}(\delta^*_{k+1})^t \le 2\left | U \right |^t\le 32\omega_{1}\left | U \right |^t.
\end{equation}

   Combining (\ref{54}), (\ref{56}) and (\ref{57}), we prove Lemma \ref{l3}.

\end{proof}
\begin{lem}\label{l4}
If condition \rm{(B)} of Theorem \rm{\ref{thm}} holds,
which implies that there exists $\omega_2>0$, such that $\bar{\alpha}_{k}\le \omega_2\cdot c_1c_2\cdots c_{k}$ for any $k\ge 1$,  then
\begin{equation*}
\mu(U)\le 32(4\omega_2+1)\left | U \right |^t.
\end{equation*}
\end{lem}

\begin{proof}
Let $k\ge k_0$ be the integer such that $\delta^*_{k+1}\le \left | U \right | <\delta^*_k$, we also distinguish the proof into two cases.

\textbf{Case 1:}  $\delta^*_{k+1}> \underline\alpha^*_{k+1}$. In this case,  by (\ref{445}), we have  $L_{k+2}+R_{k+2}<\delta^*_{k+1}$, then by (\ref{aa}), we have for any $\sigma \in D_k,  1\le l \le n_{k+1}-1$,
\begin{equation}
\begin{aligned}
\eta^*_{\sigma,l} & = \eta_{\sigma,l}+\eta_{\sigma*l,n_{k+2}}+\eta_{\sigma*(l+1),0}\\
&\le \bar{\alpha}_{k+1}+R_{k+2}+L_{k+2}   \\
&\le \omega_2c_1c_2\cdots c_{k+1}+\delta^*_{k+1} \\
&= \omega_2(L_{k+2}+\delta^*_{k+1}+R_{k+2}) +\delta^*_{k+1}\\
&\le (2\omega_2+1)\delta^*_{k+1}.\nonumber
\end{aligned}
\end{equation}
Then by the similar proof of (\ref{53}), we obtain
\begin{equation}
\delta^*_k\le 4(\omega_2+1)n_{k+1}\delta^*_{k+1}.\nonumber
\end{equation}
By the similar proof of (\ref{54}), notice that $\omega_2+1> 1$, we have $$\mu(U)\le 32(\omega_2+1)\left | U \right |^t\le 32(4\omega_2+1)\left | U \right |^t.$$

\textbf{Case 2:}  $\delta^*_{k+1}\le \underline\alpha^*_{k+1}$. In this case, by (\ref{445}), we have $\underline{\alpha}_{k+1}+L_{k+2}+R_{k+2}\ge \delta^*_{k+1}$. Then
\begin{equation}
2(\underline{\alpha}_{k+1}+L_{k+2}+R_{k+2})\ge \delta^*_{k+1}+L_{k+2}+R_{k+2}=c_1c_2\cdots c_{k+1}.\nonumber
\end{equation}
Therefore
$\underline{\alpha}_{k+1}\ge \frac{1}{4}c_1c_2\cdots c_{k+1}$ or $L_{k+2}+R_{k+2}\ge \frac{1}{4}c_1c_2\cdots c_{k+1}$.

\textbf{(a)} If $\underline{\alpha}_{k+1}\ge \frac{1}{4}c_1c_2\cdots c_{k+1}$,  then by (\ref{aa}) and (\ref{445}), we have for any $\sigma \in D_k,  1\le l \le n_{k+1}-1$,
\begin{equation}
\begin{aligned}
\eta^*_{\sigma,l} & = \eta_{\sigma,l}+\eta_{\sigma*l,n_{k+2}}+\eta_{\sigma*(l+1),0}\\
&\le \bar{\alpha}_{k+1}+R_{k+2}+L_{k+2}   \\
&\le \omega_2c_1c_2\cdots c_{k+1}+R_{k+2}+L_{k+2}\\
&\le 4\omega_2\underline{\alpha}_{k+1}+R_{k+2}+L_{k+2}\\
&\le (4\omega_2+1)(\underline{\alpha}_{k+1}+L_{k+2}+R_{k+2})\\
&=(4\omega_2+1)\underline{\alpha}^*_{k+1}.\nonumber
\end{aligned}
\end{equation}
Then we have $\bar{\alpha}^*_{k+1}\le (4\omega_2+1)\underline{\alpha}^*_{k+1}$, thus by Lemma \ref{l3}, we have $$\mu(U)\le 32(4\omega_2+1)\left | U \right |^t.$$

\textbf{(b)} If $L_{k+2}+R_{k+2}\ge \frac{1}{4}c_1c_2\cdots c_{k+1}$,  then by (\ref{aa}) and (\ref{445}), we have for any $\sigma \in D_k,  1\le l \le n_{k+1}-1$,
\begin{equation}
\begin{aligned}
\eta^*_{\sigma,l} & = \eta_{\sigma,l}+\eta_{\sigma*l,n_{k+2}}+\eta_{\sigma*(l+1),0}\\
&\le \bar{\alpha}_{k+1}+R_{k+2}+L_{k+2}   \\
&\le \omega_2c_1c_2\cdots c_{k+1}+R_{k+2}+L_{k+2}\\
&\le 4\omega_2(R_{k+2}+L_{k+2})+R_{k+2}+L_{k+2}\\
&\le (4\omega_2+1)(L_{k+2}+R_{k+2})\\
&\le (4\omega_2+1)\underline{\alpha}^*_{k+1}.\nonumber
\end{aligned}
\end{equation}
Then we also have $\bar{\alpha}^*_{k+1}\le (4\omega_2+1)\underline{\alpha}^*_{k+1}$ and $\mu(U)\le 32(4\omega_2+1)\left | U \right |^t.$
\end{proof}

\begin{lem}\label{l5}
If condition \rm{(C)} of Theorem \rm{\ref{thm}} holds,
which implies that there exists $\omega_3>0$, such that $n_{k}\underline\alpha_{k}\ge \omega_3\cdot c_1c_2\cdots c_{k-1}$ for any $k\ge 1$, then
\begin{equation}
\mu(U)\le 8\max\{1,\omega_3^{-1}\}\left | U \right |^t.
\end{equation}

\end{lem}

\begin{proof}
Let $k\ge k_0$ be the integer such that $\delta^*_{k+1}\le \left | U \right | <\delta^*_k$. Then by (\ref{42}), we have
\begin{equation}
\omega_3\delta^*_k\le \omega_3c_1c_2\cdots c_{k}\le n_{k+1}\underline{\alpha}_{k+1}\le n_{k+1}\underline{\alpha}^*_{k+1},\nonumber
\end{equation}
which implies that $\delta^*_k\le \omega_3^{-1}n_{k+1}\underline{\alpha}^*_{k+1}$.
Then as in the proof of (\ref{56}) and (\ref{57}), we obtain
$\mu(U)\le \max\{8(\omega^{-1}_{3})^{t},2\}\left| U \right|^{t}\le 8\max\{1,\omega_3^{-1}\}\left | U \right |^t$.

\end{proof}

By Lemma \ref{l3}, Lemma \ref{l4}, Lemma \ref{l5} and (1) of Lemma \ref{l1}, we prove that
$\dim_{H}E\ge t$. By the arbitrariness of $0<t<s$, we obtain  $\dim_{H}E\ge s$ and  finish the proof of Theorem \ref{thm}.

\bigskip

\section{The  Quasisymmetric  minimalities of the homogeneous Moran sets}
We begin to prove Theorem \ref{thm2}, we reconstruct the homogeneous Moran sets again.
\subsection{The second reconstruction of the homogeneous Moran sets}
\begin{lem}\label{l6}
	Let $E=E(I_{0},\left\{n_{k}\right\},\left\{c_{k}\right\})$ be a homogeneous Moran set which satisfies the conditions of Theorem {\rm\ref{thm2}}, and $E(I^*_{0},\{n^*_{k}\},\{c^*_{k}\})$ be the first reconstructed form of $E$. Then there is a sequence of closed sets  whose length is decreasing, denoted by $\{T_{m}\}_{m\ge 0}$, such that $E=\cap_{k\ge0}E_{k}=\cap_{k\ge0}E_{k}^{*}=\cap_{m\ge0}T_{m}$. Furthermore, $\{T_{m}\}_{m\ge 0}$ satisfies the following conditions:
\begin{enumerate}
  \item[\textup{(1)}] For any $m\ge0$, we have $T_{m}=\bigcup_{t=1}^{p_m} F_{t}$, where $p_m \in [1,+\infty)\cap \mathbb{Z}^{+}$,  $\{F_{t}\}_{1\le t\le p_m}$ is a sequence of close intervals, which are called the branches of $T_{m}$, and satisfying $\text{int}(F_{i_1})\cap \text{int}(F_{j_1})=\emptyset$ for any $1\le i_1<j_1\le p_m$. Denote $\mathcal{T}_{m}=\{A: A ~~~~\text{is a branch of}~~~~ T_{m}\}$;
  \item[\textup{(2)}] $\left\{E_{k}^{*}\right\}_{k\ge0}$ is a subsequence of $\left\{T_{m}\right\}_{m\ge 0}$, where $T_{m_{k}}=E_{k}^{*}$ for any $k\ge 0$;

  \item[\textup{(3)}] If $E$ satiefies the condition {\rm(A)} of Theorem {\rm\ref{thm2}}, then there exists $M \in \mathbb{Z}^{+}$ with $M>2\omega_1$ such that each branch of $T_{m-1}$ contains at most $M^{2}$ branches of $T_{m}$ for any $m\ge1$; if $E$ satiefies the condition {\rm(B)} of Theorem {\rm\ref{thm2}}, then there exists $M \in \mathbb{Z}^{+}$ with $M>2(\omega_2+1)$ such that each branch of $T_{m-1}$ contains at most $M^{2}$ branches of $T_{m}$ for any $m\ge1$, where $\omega_1$, $\omega_2$ are the constants in Theorem {\rm\ref{thm2}};
		\item[\textup{(4)}]If $E$ satiefies the condition {\rm(A)} of Theorem {\rm\ref{thm2}}, then $\max_{I\in \mathcal{T}_m}\left | I \right |\le 2 \omega_1\min_{I\in \mathcal{T}_m}\left | I \right |$, and if $E$ satiefies the condition {\rm(B)} of Theorem {\rm\ref{thm2}}, then
$\max_{I\in \mathcal{T}_m}\left | I \right |\le 2 (\omega_2+1)\min_{I\in \mathcal{T}_m}\left | I \right |$ for any $m\ge0$.
\end{enumerate}
\end{lem}
\begin{proof}
First, we construct $\{T_{m}\}_{m\ge 0}$.

Let $M=\min\{A_1:A_1>2\omega_1, A_1\in \mathbb{Z}^{+}\}$ if $E$ satisfies the conditon (A)  of Theorem {\rm\ref{thm2}}, and let $M=\min\{A_1:A_1>2(\omega_2+1),A_1\in \mathbb{Z}^{+}\}$ if $E$ satisfies the conditon (B)  of Theorem {\rm\ref{thm2}}, for any $k\ge 1$, $i_{k}\in \mathbb{Z}^{+}$ satisfies following conditions:
\begin{enumerate}
 \item[\textup{(\rmnum{1})}]$i_{k}=1$ when $2\le n_{k}^{*}<M$;
 \item[\textup{(\rmnum{2})}]$i_{k}$ satisfies $M^{i_{k}}\le n_{k}^{*}<M^{i_{k}+1}$ when $n_{k}^{*}\ge M$.
\end{enumerate}
Let $m_{0}=0,\ m_{k}=\sum_{l=1}^{k}i_{l},$ then $m_k=m_{k-1}+i_k.$
	
	For any $k\ge 0$, let $T_{m_{k}}=E_{k}^{*}$  and  $\mathcal{T}_{m_{k}}=\{I_{\omega}^{\ast}:\omega\in D_{k}\}$,
then $T_{m_{k}}$ is the union of all $k$-order first reconstructed basic intervals of $E$.
Next, we construct $T_{m}$ for any $k\ge 1$ and $m_{k-1}<m<m_{k}$.

\begin{enumerate}
		\item[\textup{(1)}] If $M\le n_{k}^{*}< M^{2}$, then $i_{k}=1$ and $m_k=m_{k-1}+1$, there is no integer $m$ which satisfies $m_{k-1}< m<m_k$.
		\item[\textup{(2)}] If $n_{k}^{*}\ge M^{2}$, then $i_{k}\ge 2$, and there exists  $b_{j}\in\left\{0,1,\cdots,M-1\right\}$  for any $j\in\left\{0,1,\cdots,i_{k}-1\right\}$, such that
\begin{equation*}
		n_{k}^{*}=b_{0}+b_{1}M+b_{2}M^{2}+\cdots+b_{i_{k}-1}M^{i_{k}-1}+M^{i_{k}}.
\end{equation*}
	
For any $k\ge1$ and $\sigma\in D_{k-1}$, since $T_{m_{k-1}}=E_{k-1}^{*}$, then $T_{m_{k-1}}$ has $N_{k-1}^{*}$ branches and
$I_{\sigma}^{*}$ contains $n_{k}^{*}$ $k$-order first reconstructed basic intervals of $E$ for any $I_{\sigma}^{*}\in \mathcal{T}_{m_{k-1}}$, which are  $I_{\sigma*1}^{*},\cdots,I_{\sigma*n_{k}^{*}}^{*}$  from left to right.

Next, we construct $T_{m_{k-1}+i}$ for any  $1\le i\le i_k-1$.

For $t$ closed intervals $Q_{1},Q_{2},\cdots,Q_{t}$, let $[Q_{1},Q_{2},\cdots,Q_{t}]$ be the smallest closed interval which contains them.

\begin{enumerate}
\item[\textup{(a)}] For any $I_{\sigma}^{*}\in \mathcal{T}_{m_{k-1}}$,
let $n^*_k=Md_1+b_0=b_0(d_1+1)+(M-b_0)d_1$ where
$d_{1}=b_{1}+b_{2}M+\cdots+b_{i_{k}-1}M^{i_{k}-2}+M^{i_{k}-1}$.
We define some subintervals of $I_{\sigma}^{*}$ as follows,
$$I_{1}^{\sigma,1}=[I^*_{\sigma*1},\cdots,I^*_{\sigma*(d_1+1)}],$$
$$I_{2}^{\sigma,1}=[I^*_{\sigma*(d_1+2)},\cdots,I^*_{\sigma*(2d_1+2)}],$$
$$\cdots$$
$$I_{b_0}^{\sigma,1}=[I^*_{\sigma*\big((b_0-1)(d_1+1)+1\big)} ,\cdots,I^*_{\sigma*\big(b_0(d_1+1)\big)}],$$
$$I_{b_0+1}^{\sigma,1}=[I^*_{\sigma*\big(b_0(d_1+1)+1\big)},\cdots,I^*_{\sigma*\big(b_0(d_1+1)+d_1\big)}],$$
$$I_{b_0+2}^{\sigma,1}=[I^*_{\sigma*\big(b_0(d_1+1)+d_1+1\big)},\cdots,I^*_{\sigma*\big(b_0(d_1+1)+2d_1\big)}],$$
$$\cdots$$
$$I_{M}^{\sigma,1}=[I^*_{\sigma*(n^*_k+1-d_1)},\cdots,I^*_{\sigma*n^*_k}].$$

Then each $I_{1}^{\sigma,1}, \cdots, I_{b_0}^{\sigma,1}$ contains $d_1+1$   $k$-order first reconstructed basic intervals of $E$,
and each $I_{b_0+1}^{\sigma,1}, \cdots, I_{M}^{\sigma,1}$ contains $d_1$  $k$-order first reconstructed basic intervals of $E$.
Let $T_{m_{k-1}+1}=\bigcup_{\sigma \in D_{k-1}}\bigcup_{i=1}^{M}I^{\sigma,1}_i$, and the $M$ closed intervals $I_{1}^{\sigma,1},\cdots,I_{M}^{\sigma,1}$ be the $M$ branches of $T_{m_{k-1}+1}$ in $I_{\sigma}^{*}$ , then each branch of $T_{m_{k-1}}$ contains $M$ branches of $T_{m_{k-1}+1}$.

\item[\textup{(b)}] If $i_{k}=2$, then $m_{k}=m_{k-1}+2$. We have defined $T_{m_{k-1}+1}$ as above, and $T_{m_{k-1}}=E_{k-1}^{*}$, $T_{m_{k}}=E_{k}^{*}$.
    Thus we finish the construction of $T_{m_{k-1}+i}$  for any $1\le i\le i_k-1$.

\item[\textup{(c)}] If $i_{k}\ge3$, we need to construct  $T_{m_{k-1}+2}$. Let $d_2=b_2+b_3M+\cdots +b_{i_k-1}M^{i_k-3}+M^{i_k-2}$, then $d_{1}=Md_{2}+b_{1}$, $n^*_k=M^2d_2+b_1M+b_0=b_0(Md_2+b_1+1)+(M-b_0)(Md_2+b_1)$.

For any $I^{\sigma,1}_{i}\in \mathcal{T}_{m_{k-1}+1}(\sigma\in D_{k-1}, 1\le i\le M)$, we consider the following two cases:

\textbf{(c1):} If $1\le i\le b_{0}$,
then each $I^{\sigma,1}_{i}$ contains $d_1+1$  $k$-order first reconstructed basic intervals of $E$,
where $d_1+1=Md_2+b_1+1=(d_2+1)(b_1+1)+d_2(M-b_1-1)$.
Since $I_{i}^{\sigma,1}=[I^*_{\sigma*\big((i-1)d_1+i\big)}, I^*_{\sigma*\big(i(d_1+1)\big)}]$,  we define
$$I_{i*1}^{\sigma,1}=[I^*_{\sigma*\big((i-1)d_1+i\big)},\cdots,I^*_{\sigma*\big((i-1)d_1+i+d_2\big)}],$$
$$I_{i*2}^{\sigma,1}=[I^*_{\sigma*\big((i-1)d_1+i+d_2+1\big)},\cdots,I^*_{\sigma*\big((i-1)d_1+i+2d_2+1\big)}],$$
$$\cdots$$
$$I_{i*(b_1+1)}^{\sigma,1}=[I^*_{\sigma*\big((i-1)d_1+i+b_1d_2+b_1\big)},\cdots,I^*_{\sigma*\big((i-1)d_1+i+(b_1+1)d_2+b_1\big)}],$$
$$I_{i*(b_1+2)}^{\sigma,1}=[I^*_{\sigma*\big((i-1)d_1+i+(b_1+1)(d_2+1)\big)},\cdots,I^*_{\sigma*\big((i-1)d_1+i+(b_1+1)(d_2+1)+d_2-1\big)}],$$
$$\cdots$$
$$I_{i*M}^{\sigma,1}=[I^*_{\sigma*(id_1+i+1-d_2)},\cdots,I^*_{\sigma*\big(i(d_1+1)\big)}].$$

Then each $I_{i*1}^{\sigma,1}, \cdots, I_{i*(b_1+1)}^{\sigma,1}$ contains $d_2+1$  $k$-order first reconstructed basic intervals of $E$,
and each  $I_{i*(b_1+2)}^{\sigma,1}, \cdots, I_{i*M}^{\sigma,1}$ contains $d_2$  $k$-order first reconstructed basic intervals of $E$.

\textbf{(c2):} If $b_{0}+1\le i\le M$, then
each $I^{\sigma,1}_{i}$ contains $d_1$ $k$-order first reconstructed basic intervals of $E$,
where $d_1=Md_2+b_1=(d_2+1)b_1++d_2(M-b_1)$.
Since $I_{i}^{\sigma,1}=[I^*_{\sigma*\big(b_0(d_1+1)+(i-b_0-1)d_1+1\big)}, I^*_{\sigma*\big(b_0(d_1+1)+(i-b_0)d_1\big)}]$,
we define
$$I_{i*1}^{\sigma,1}=[I^*_{\sigma*\big(b_0(d_1+1)+(i-b_0-1)d_1+1\big)},\cdots,I^*_{\sigma*\big((i-1)d_1+b_0+1+d_2\big)}],$$
$$I_{i*2}^{\sigma,1}=[I^*_{\sigma*\big((i-1)d_1+b_0+d_2+2\big)},\cdots,I^*_{\sigma*\big((i-1)d_1+b_0+2d_2+2\big)}],$$
$$\cdots$$
$$I_{i*b_1}^{\sigma,1}=[I^*_{\sigma*\big((i-1)d_1+b_0+(b_1-1)d_2+b_1\big)},\cdots,I^*_{\sigma*\big((i-1)d_1+b_0+b_1d_2+b_1\big)}],$$
$$I_{i*(b_1+1)}^{\sigma,1}=[I^*_{\sigma*\big((i-1)d_1+b_0+b_1d_2+b_1+1\big)},\cdots,I^*_{\sigma*\big((i-1)d_1+b_0+b_1d_2+b_1+d_2\big)}],$$
$$\cdots$$
$$I_{i*M}^{\sigma,1}=[I^*_{\sigma*(id_1+b_0+1-d_2)},\cdots,I^*_{\sigma*\big(b_0(d_1+1)+(i-b_0)d_1\big)}].$$

Then each $I_{i*1}^{\sigma,1}, \cdots, I_{i*b_1}^{\sigma,1}$ contains $d_2+1$ $k$-order first reconstructed basic intervals of $E$,
and each $I_{i*(b_1+1)}^{\sigma,1}, \cdots, I_{i*M}^{\sigma,1}$ contains $d_2$   $k$-order first reconstructed basic intervals of $E$.

For any $1\le l \le M$ and$(l-1)M+1\le h\le lM$, define $I^{\sigma,2}_{h}=I^{\sigma,1}_{l*\big(h-(l-1)M\big)}$. Let $$T_{m_{k-1}+2}=\bigcup_{\sigma \in D_{k-1}}\bigcup_{i=1}^{M} \bigcup_{j=1}^{M}I^{\sigma,1}_{i*j}=\bigcup_{\sigma \in D_{k-1}}\bigcup_{h=1}^{M^2}I^{\sigma,2}_{h},$$ and let the $M$ closed intervals $I^{\sigma,1}_{i*1},I^{\sigma,1}_{i*2},\cdots, I^{\sigma,1}_{i*M}$  be the $M$ branches of $T_{m_{k-1}+2}$ in $I_{i}^{\sigma,1}$,  then each branch of $T_{m_{k-1}+1}$ contains $M$ branches of $T_{m_{k-1}+2}$.

\item[\textup{(d)}] If $i_{k}=3$, then $m_{k}=m_{k-1}+3$. We have defined $T_{m_{k-1}+1}$, $T_{m_{k-1}+2}$ as above, $T_{m_{k-1}}=E_{k-1}^{*}$, $T_{m_{k}}=E_{k}^{*}$. Then the construction is done.

\item[\textup{(e)}]If $i_k\ge 4$, then we have $m_k=m_{k-1}+i_k$. If we finish the construction of $T_{m_{k-1}+i-1}(3\le i\le i_k-1)$,  then we repeat the method of the construction of $T_{m_{k-1}+i-1}$ from $T_{m_{k-1}+i-2}$ to construct $T_{m_{k-1}+i}$ from $T_{m_{k-1}+i-1}$. Then each branch of $T_{m_{k-1}+j-1}\ (1\le j\le i_{k}-1)$ contains $M$ branches of $T_{m_{k-1}+j}$, which implies that each branch of $T_{m_{k-1}}$ contains $M^{i_{k}-1}$ branches of $T_{m_{k-1}+i_{k}-1}$.
    Notice that  $m_k=m_{k-1}+i_k$ and $T_{m_{k}}=E_{k}^{*}$ for any $k\ge 0$, we  obtain  that each branch of $T_{m_{k-1}}$ contains $n_k^{*}$ branches of $T_{m_{k}}$, then the number of the branches of $T_{m_{k}}$ contained in each branch of $T_{m_{k-1}+i_{k}-1}$ is at most $M^{2}$(If  there is a branch of $T_{m_{k-1}+i_{k}-1}$ containing $M^{'}$ branches of $T_{m_{k}}$ with $M^{'}>M^{2}$, then
    the number of the branches of $T_{m_{k}}$ contained in any branch of $T_{m_{k-1}+i_{k}-1}$
    is $M^{'}$, $M^{'}+1$ or $M^{'}-1$. We obtain that $n_{k}^{*}\ge M^{2}\times M^{i_{k}-1}=M^{i_k+1}$, which is contrary to the fact $ n_k^{*}<M^{i_k+1}$).

\end{enumerate}

    We finish the construction of $\{T_{m}\}_{m\ge 0}$ and prove that the conditions (1)-(3) of Lemma \ref{l6}  hold if the conditions of Theorem \ref{thm2} is satisfied.

Now we consider the relationship of the lengths of the branches.

Since $T_{m_{k}}=E_{k}^{*}$ for any $k\ge 0$, we have $\max_{I\in\mathcal{T}_{m_{k}}}\left|I\right|= \min_{I\in\mathcal{T}_{m_{k}}}\left|I\right|$ for any $k \ge 0$.

For any $k\ge 1$, $m_{k-1}< m <m_k$ and $I\in T_m$, let
$\Psi (I,T_{m_{k}})=\#(\{I'\in T_{m_{k}}: I'\subset I\})$($\#$ denotes the cardinality),  which  means $\Psi (I,T_{m_{k}}) $ is the number of the branches of $T_{m_{k}}$contained in $I$($\Psi (I,T_{m_{k}}) $ is also the number of the  $k$-order first reconstructed basic intervals of $E$ contained in $I$),
then we have \begin{equation}\label{61}\Psi (\max_{I\in\mathcal{T}_{m} }\left | I \right |,T_{m_{k}})\le \Psi (\min_{I\in\mathcal{T}_{m} }\left | I \right |,T_{m_{k}})+1\end{equation} from the above construction.

If  $E$ satisfies the condition (A) of Theorem \ref{thm2}, then for any $k\ge 1$, we have
\begin{equation}
\bar\alpha_k\le  \omega_{1}\underline\alpha_k.\nonumber
\end{equation}
By (\ref{52}), we obtain
\begin{equation}\label{62}
\bar\alpha^*_k\le \omega_{1} \underline\alpha^*_k.
\end{equation}

Combing (\ref{61}),  (\ref{62}) and $M>2\omega_{1}\ge2$, we obtain
\begin{equation}
\begin{aligned}
\max_{I\in\mathcal{T}_{m} }\left | I \right |&\le \Psi (\max_{I\in\mathcal{T}_{m} }\left | I \right |,T_{m_{k}})\delta^*_k+(\Psi (\max_{I\in\mathcal{T}_{m} }\left | I \right |,T_{m_{k}})-1)\bar{\alpha}^*_k\\ &\le (\Psi (\min_{I\in\mathcal{T}_{m} }\left | I \right |,T_{m_{k}})+1)\delta^*_k+\Psi (\min_{I\in\mathcal{T}_{m} }\left | I \right |,T_{m_{k}})\bar{\alpha}^*_k\\
&\le 2\omega_{1}[\Psi (\min_{I\in\mathcal{T}_{m} }\left | I \right |,T_{m_{k}})\delta^*_k+(\Psi (\min_{I\in\mathcal{T}_{m} }\left | I \right |,T_{m_{k}})-1)\underline{\alpha}^*_k]\\
&\le 2\omega_{1} \min_{I\in\mathcal{T}_{m} }\left | I \right |.\nonumber
\end{aligned}
\end{equation}

	\end{enumerate}

If  $E$ satisfies the condition (B) of Theorem \ref{thm2}, then for any $k\ge 1$, we have
\begin{equation}
\bar\alpha_k\le \omega_{2} \delta_k.\nonumber
\end{equation}
By (\ref{444}), \begin{equation}\label{64}
\bar{\alpha}^*_k=\bar{\alpha}_k+L^*_k+R^*_k.
\end{equation}

Combing (\ref{447}), (\ref{61}),  (\ref{64}) and $M>2(\omega_{2}+1)>2$, we obtain
\begin{equation}
\begin{aligned}
\max_{I\in\mathcal{T}_{m} }\left | I \right |&\le \Psi (\max_{I\in\mathcal{T}_{m} }\left | I \right |,T_{m_{k}})\delta^*_k+(\Psi (\max_{I\in\mathcal{T}_{m} }\left | I \right |,T_{m_{k}})-1)\bar{\alpha}^*_k\\&\le\Psi (\max_{I\in\mathcal{T}_{m} }\left | I \right |,T_{m_{k}})\delta^*_k+(\Psi (\min_{I\in\mathcal{T}_{m} }\left | I \right |,T_{m_{k}})(\bar{\alpha}_k+L^*_k+R^*_k)\\&\le \Psi (\max_{I\in\mathcal{T}_{m} }\left | I \right |,T_{m_{k}})\delta^*_k+(\Psi (\min_{I\in\mathcal{T}_{m} }\left | I \right |,T_{m_{k}})(\omega_{2} \delta_k+L^*_k+R^*_k)\\&=\Psi (\max_{I\in\mathcal{T}_{m} }\left | I \right |,T_{m_{k}})\delta^*_k+(\Psi (\min_{I\in\mathcal{T}_{m} }\left | I \right |,T_{m_{k}})(\omega_{2} (\delta^*_k+L^*_k+R^*_k)+L^*_k+R^*_k) \\
&\le (\omega_{2}+1)[(\Psi (\min_{I\in\mathcal{T}_{m} }\left | I \right |,T_{m_{k}})+1)\delta^*_k+\Psi (\min_{I\in\mathcal{T}_{m} }\left | I \right |,T_{m_{k}})(L^*_k+R^*_k)]\\
&\le 2(\omega_{2}+1)[\Psi (\min_{I\in\mathcal{T}_{m} }\left | I \right |,T_{m_{k}})\delta^*_k+(\Psi (\min_{I\in\mathcal{T}_{m} }\left | I \right |,T_{m_{k}})-1)(L^*_k+R^*_k)]\\
&\le 2(\omega_{2}+1)[\Psi (\min_{I\in\mathcal{T}_{m} }\left | I \right |,T_{m_{k}})\delta^*_k+(\Psi (\min_{I\in\mathcal{T}_{m} }\left | I \right |,T_{m_{k}})-1)\underline{\alpha}^*_k]\\
&\le 2(\omega_{2}+1)\min_{I\in\mathcal{T}_{m} }\left | I \right |.\nonumber
\end{aligned}
\end{equation}

Thus we prove that the condition (4) of Lemma \ref{l6}  hold if the conditions of Theorem \ref{thm2} is satisfied.

We complete the constructions of  $\{T_m\}_{m\ge0}$ and the proof of Lemma \ref{l6}.

\end{proof}
\bigskip

\begin{rem}
Without loss of generality, we  assume that $I_{0}^{*}=[0,1]$, then $\delta^*_{0}=1$ and $T_{m_{0}}=E_{0}^{*}=[0,1]$.
\end{rem}
\bigskip
We have the following lemma.

\begin{lem}\label{l7}
	Let $E=E(I_{0},\left\{n_{k}\right\},\left\{c_{k}\right\})$ be a homogeneous Moran set which satisfies the conditions of Theorem \ref{thm2}, $\left\{T_{m}\right\}_{m\ge0}$ be the sequences in Lemma \text{\rm\ref{l6}},
$l(T_{m})$ be the total length of all branches of $T_{m}$.
Then for any $k\ge 1$,
\begin{equation}
l(T_{m_{k}})=N^*_k\delta^*_k.
\end{equation}
Furthermore, if $E$ satiefies the condition {\rm(A)} of Theorem {\rm\ref{thm2}}, then for any $k\ge 1$ and $m_{k-1}<m<m_k$,
\begin{equation}
(1-\frac{2\omega_1}{M})N^*_{k-1}\delta^*_{k-1}\le l(T_m) \le N^*_{k-1}\delta^*_{k-1}.
\end{equation}
If $E$ satiefies the condition {\rm(B)} of Theorem {\rm\ref{thm2}}, then for any $k\ge 1$ and $m_{k-1}<m<m_k$,
\begin{equation}
(1-\frac{2(\omega_2+1)}{M})N^*_{k-1}\delta^*_{k-1}\le l(T_m) \le N^*_{k-1}\delta^*_{k-1}.
\end{equation}

\end{lem}

\begin{proof}
	Since for any $k\ge 1$, $T_{m_{k}}=E_{k}^{*}$, we have $l(T_{m_{k}})=l(E_{k}^{*})=N_{k}^{*}\delta_{k}^{*}$.
Notice that $\{l(T_{m})\}_{m\ge0}$ is decreasing, then
$l(T_{m})\le l(T_{m_{k-1}})=l(E_{k-1}^{*})=N_{k-1}^{*}\delta_{k-1}^{*}$ for any $k\ge 1$ and $m_{k-1}<m<m_{k}$.

    So we only need to prove that if $E$ satiefies condition {\rm(A)} of Theorem {\rm\ref{thm2}}, then
    $(1-\frac{2\omega_1}{M})N^*_{k-1}\delta^*_{k-1}\le l(T_m)$  for any $k\ge 1$ and $m_{k-1}<m<m_{k}$, if $E$ satiefies condition {\rm(B)} of Theorem {\rm\ref{thm2}}, then
    $(1-\frac{2(\omega_2+1)}{M})N_{k-1}^{*}\delta_{k-1}^{*}\le l(T_{m})$ for any $k\ge 1$ and $m_{k-1}<m<m_{k}$.

 By the construction of $\left\{T_{m}\right\}_{m\ge0}$, in order to get $T_{m_{k}-1}$ from $T_{m_{k-1}}$, we should remove a left-closed and right-open interval of length $L^*_k$ and a left-open and  right-closed interval of length  $R^*_k$ from each branch of $T_{m_{k-1}}$, and remove $[\sum_{j=0}^{i_{k}-2}M^{j}(M-1)]N_{k-1}^{*}=(M^{i_{k}-1}-1)N_{k-1}^{*}$ open intervals whose lengths are at most $\overline{\alpha}_{k}^{*}$ from  $E_{k-1}^{*}=T_{m_{k-1}}$.

 If $E$ satiefies the condition {\rm(A)} of Theorem {\rm\ref{thm2}}, notice that $n^*_k\ge 2$, $\omega_1\ge 1$ and
$M^{i_k}\le n^*_k <M^{i_k+1}$, then by ($\ref{447}$) and ($\ref{52}$), we have
\begin{equation}
\begin{aligned}
l(T_{m_k-1})&\ge N^*_{k-1}\delta^*_{k-1}-N^*_{k-1}[(L^*_k+R^*_k)+(M^{i_k-1}-1)\bar{\alpha}^*_k]\\
&\ge N^*_{k-1}\delta^*_{k-1}-M^{i_k-1}N^*_{k-1}\bar{\alpha}^*_k\\
&\ge N^*_{k-1}\delta^*_{k-1}-\frac{n^*_k}{M}N^*_{k-1}\bar{\alpha}^*_k\\
&\ge N^*_{k-1}\delta^*_{k-1}-\frac{2(n^*_k-1)}{M}N^*_{k-1}\bar{\alpha}^*_k\\
&\ge N^*_{k-1}\delta^*_{k-1}-\frac{2\omega_1}{M}N^*_{k-1}(n^*_k-1)\underline{\alpha}^*_k\\
&\ge N^*_{k-1}\delta^*_{k-1}-\frac{2\omega_1}{M}N^*_{k-1}\delta^*_{k-1}\\
&\ge (1-\frac{2\omega_1}{M})N^*_{k-1}\delta^*_{k-1}.\nonumber
\end{aligned}
\end{equation}
Since $\left\{T_{m}\right\}_{m\ge0}$ is a sequence whose length is decreasing,
we obtain
\begin{equation*}
l(T_{m})\ge l(T_{m_k-1})\ge (1-\frac{2\omega_1}{M})N^*_{k-1}\delta^*_{k-1}
\end{equation*}
for any $k\ge 1$ and $m_{k-1}<m<m_{k}$.

If $E$ satiefies the condition {\rm(B)} of Theorem {\rm\ref{thm2}}, notice that
$M^{i_k}\le n^*_k <M^{i_k+1}$ and $1+\omega_2\ge 1$, then by ($\ref{444}$) and ($\ref{447}$), we have
\begin{equation}
\begin{aligned}
l(T_{m_k-1})&\ge N^*_{k-1}\delta^*_{k-1}-N^*_{k-1}[(L^*_k+R^*_k)+(M^{i_k-1}-1)\bar{\alpha}^*_k]\\
&\ge N^*_{k-1}\delta^*_{k-1}-M^{i_k-1}N^*_{k-1}\bar{\alpha}^*_k\\
&= N^*_{k-1}\delta^*_{k-1}-M^{i_k-1}N^*_{k-1}(\bar{\alpha}_k+L^*_k+R^*_k)\\
&\ge N^*_{k-1}\delta^*_{k-1}-M^{i_k-1}N^*_{k-1}(\omega_2\delta_k+L^*_k+R^*_k)\\
&\ge N^*_{k-1}\delta^*_{k-1}-M^{i_k-1}N^*_{k-1}(\omega_2+1)(\delta_k+L^*_k+R^*_k)\\
&\ge N^*_{k-1}\delta^*_{k-1}-\frac{n^*_k}{M}N^*_{k-1}(\omega_2+1)(\delta^*_k+2(L^*_k+R^*_k))\\
&\ge N^*_{k-1}\delta^*_{k-1}-\frac{2(\omega_2+1)}{M}N^*_{k-1}\delta^*_{k-1}\\
&= (1-\frac{2(\omega_2+1)}{M})N^*_{k-1}\delta^*_{k-1},\nonumber
\end{aligned}
\end{equation}
which implies that
\begin{equation*}
l(T_m)\ge  l(T_{m_{k}-1}) \ge(1-\frac{2(\omega_2+1)}{M})N^*_{k-1}\delta^*_{k-1}
\end{equation*}
for any $k\ge 1$ and $m_{k-1}<m<m_{k}$ and we complete the proof of Lemma \ref{l7}.
\end{proof}

\bigskip

\subsection{Some marks and lemmas}

Let $E=E(I_{0},\left\{n_{k}\right\},\left\{c_{k}\right\})$ be a homogeneous Moran set which satisfies the conditions of Theorem \ref{thm2}, $\left\{T_{m}\right\}_{m\ge0}$ be the sequences in Lemma \text{\rm\ref{l6}}, $f$ is a 1-dimensional quasisymmetric mapping.

For any $m\ge0$ and $I\in T_m$, $I-(I\cap T_{m+1})$ consists of some intervals which interiors are disjoint (the most left and right  intervals  are half-open and half-closed intervals or empty sets, others are open intervals), we call them the gaps of $I$. The collection of all the gaps of all the branchs of $T_m$ is denoted by  $\mathcal{G}_m$, which implies that $\mathcal{G}_m=\{$The gaps of $I: I\in \mathcal{T}_m\}$.

For any  $I\in T_m$, we denote $\mathcal{G}(I)=\{L:L\subset I, L\in\mathcal{G}_m\}$.
According to the reconstruction process, for any $I\in T_m$, $I$ contains at most $M^2$ branches of $T_{m+1}$, then $\#(\mathcal{G}(I))\le M^2+1$($\#$ denotes the cardinality).

For any $m\ge 1$ and $I\in T_m$ , denote the branch of $T_{m-1}$ which contain $I$ by $Xa(I)$.

For any  $m\ge1$, $k\ge 1$ and $m_{k-1}< m\le m_{k}$ , denote
$$\Lambda^*(m)=\frac{\max_{I\in \mathcal{T}_m }\left | I \right | }{\min_{I\in \mathcal{T}_{m-1} }\left | I \right |},\quad \Lambda_*(m)=\frac{\min_{I\in \mathcal{T}_m }\left | I \right | }{\max_{I\in \mathcal{T}_{m-1} }\left | I \right |},
 $$

$$\Gamma^*(m)=\frac{\bar{\alpha}^*_k}{\min_{I\in \mathcal{T}_{m-1} }\left | I \right |},\quad
\Gamma_*(m)=\frac{\underline{\alpha}^*_k}{\max_{I\in \mathcal{T}_{m-1} }\left | I \right |},
 $$

$$\beta_m=\max\{\frac{\left | F \right |}{\left | I \right |}, I\in T_m, F\in \mathcal{G}(I)\},$$
$$\Theta_m=\min \{\frac{\sum_{i=1}^{N(I_m)} \left | I_{m,i} \right |}{\left | I_m \right | }:I_m\in T_m\},$$

$$\chi_m=\max\{\frac{\left | I \right | }{\left | Xa(I) \right | }:I\in T_m\},$$
where $N(I_{m})$ is the number of the branches of $T_{m+1}$ contained in $I_{m}$, then $N(I_{m})\le M^{2}$.

We have the following lemmas.
\begin{lem}\label{l8}
For any $m\ge0$, we have $\Theta_m\ge 1-(M^2+1)\beta_m$.
\end{lem}
\begin{proof}
For $I_m\in \mathcal{T}_m$, we have
$\beta_m\ge\frac{\left | F \right | }{\left | I_m \right | } $ for any $F\in \mathcal{G}(I_m)$. Then
\begin{equation*}
\sum_{F\in \mathcal{G}(I_m) } \frac{\left | F \right | }{\left | I_m \right | } \le \sum_{F\in \mathcal{G}(I_m) }\beta_m\le(M^2+1)\beta_m.
\end{equation*}
Which implies that
\begin{equation*}
\frac{\sum_{i=1}^{N(I_m)}\left | I_{m,i} \right | }{\left | I_m \right | }=\frac{\left | I_{m} \right |- \sum_{F\in \mathcal{G}(I_m)} \left | F \right |   }{\left | I_m \right | }\ge1-(M^2+1)\beta_m.
\end{equation*}
By the arbitrariness of $I_m$ and $\tilde{I_m}$, we have
\begin{equation*}
\Theta_m=\min \{\frac{\sum_{i=1}^{N(I_m)} \left | I_{m,i} \right |}{\left | I_m \right | }:I_m\in T_m\}\ge 1-(M^2+1)\beta_m.
\end{equation*}

\end{proof}

\begin{lem}\label{l9}
Suppose $\{w_m\}_{m\ge 0}$ is a sequence of non-negative real numbers, and
\begin{equation*}
\lim_{m \to \infty} \frac{1}{m} \sum_{i=0}^{m-1} w_i=0.
\end{equation*}
Then we have
\begin{equation*}
\lim_{m \to \infty} \frac{V(m,\varepsilon )}{m} =1,
 \end{equation*}
for any $\varepsilon\in (0,1)$, where $V(m,\varepsilon )=\#(\{0\le i\le m-1: w_i<\varepsilon\})$$(\#$ denotes the cardinality$)$.
\end{lem}

\begin{proof}
Notice that \begin{equation}\label{111}
\sum_{i=0}^{m-1} w_i=\sum_{w_i<\varepsilon, 0\le i\le m-1}w_i+\sum_{w_i\ge \varepsilon, 0\le i\le m-1}w_i\ge \sum_{w_i\ge \varepsilon, 0\le i\le m-1}w_i \ge \big(m-V(m,\varepsilon ) \big)\varepsilon.\end{equation}
Since
\begin{equation*}
\lim_{m \to \infty} \frac{1}{m} \sum_{i=0}^{m-1} w_i=0,
\end{equation*}
we have
\begin{equation*}
1\ge\limsup\limits_{m \to \infty} \frac{V(m,\varepsilon )}{m}\ge\liminf\limits_{m \to \infty} \frac{V(m,\varepsilon )}{m} =1-\limsup\limits_{m \to \infty} \frac{m-V(m,\varepsilon )}{m}\ge1-\limsup\limits_{m \to \infty}\frac{1}{m\varepsilon }\sum_{i=0}^{m-1}w_i=1,
\end{equation*}
which impies that $\lim_{m \to \infty} \frac{V(m,\varepsilon )}{m} =1$.
\end{proof}

\begin{lem}\label{l10}

Let $E=E(I_{0},\left\{n_{k}\right\},\left\{c_{k}\right\})$ be a homogeneous Moran set which satisfies the conditions of Theorem {\rm\ref{thm2}}, $\{T_{m}\}_{m\ge 0}$ be the sequences of Lemma {\rm\ref{l6}}.

If $\dim_{H}E=1$, then we have
\begin{enumerate}
\item[\textup{(1)}]$\lim_{m\to \infty}\frac{\log_{M}l(T_{m}) }{m}=0$, where $l(T_{m})$ is the total length of all branches of $T_{m}$;
\item[\textup{(2)}]$\lim_{m\to \infty}\frac{\sum_{j=0}^{m-1}\beta_{j}}{m}=0$;
\item[\textup{(3)}]$\lim_{m \to \infty} \frac{1}{m} \sum_{j=0}^{m-1}\log\Theta_j=0$;
\item[\textup{(4)}]there exists $\alpha\in (0,1)$, such that $\lim_{m\to \infty}  \frac{\# S(m,\alpha )}{m}=1$, where $S(m,\varepsilon)=\{1\le i\le m:\chi_i<\varepsilon\}$ for any $\varepsilon\in (0,1)$.
\end{enumerate}

\end{lem}

\begin{proof}
(1)
For any $m\ge 1$, if there exists $k\ge1$, such that $m=m_{k}$, then $l(T_m )=l(T_{m_k}) =l(E^*_k)=n^*_1n^*_2\cdots n^*_k\delta^*_k$. Since $\delta^*_kn^*_1n^*_2\cdots n^*_k=l(E^*_k)\le 1$, we have $\frac{\log_M(n^*_1n^*_2\cdots n^*_k)}{-\log_M\delta^*_k}\le 1.$ Notice that $\dim_{H}E=1$, by Theorem \ref{thm}, we have
\begin{equation}\begin{aligned}
1\ge\limsup\limits_{k \to\infty}  \frac{\log_M n^*_1n^*_2\cdots n^*_k}{-\log_M\delta^*_k}&\ge\liminf\limits_{k \to\infty}  \frac{\log_M n^*_1n^*_2\cdots n^*_k}{-\log_M\delta^*_k}\\& =\liminf\limits_{k \to\infty}  \frac{\log_M n_1n_2\cdots n_k}{-\log_M(\delta_k-L_{k+1}-R_{k+1})}\\& =\dim_{H}E=1.
\end{aligned}
\end{equation}
Which implies that \begin{equation}\label{101}
\lim_{k \to\infty}  \frac{\log_M n^*_1n^*_2\cdots n^*_k}{-\log_M\delta^*_k}=1.
\end{equation}

Since $\log_Mn^*_j\le i_j+1$, $m_k=i_1+i_2+\cdots+i_k$ and $m_k\ge k$ for any $k\ge 1$ and $1\le j \le k$, by (\ref{101}), we have
\begin{equation*}
\begin{aligned}
0\ge\lim_{k \to \infty} \frac{\log_M n^*_1n^*_2\cdots n^*_k\delta^*_k}{m_k} & = \lim_{k \to \infty}\frac{\log_M(n^*_1n^*_2\cdots n^*_k)}{m_k}\frac{\log_M(n^*_1n^*_2\cdots n^*_k)+\log_M\delta^*_k}{\log_M(n^*_1n^*_2\cdots n^*_k)}\\
&\ge \lim_{k \to \infty}2[1-( \frac{\log_M n^*_1n^*_2\cdots n^*_k}{-\log_M\delta^*_k})^{-1}]
=0.
\end{aligned}
\end{equation*}
Which implies that \begin{equation}\label{102}
\lim_{k \to \infty}\frac{\log_{M}l(T_{m_k}) }{m_k}=\lim_{k \to \infty} \frac{\log_M n^*_1n^*_2\cdots n^*_k\delta^*_k}{m_k}=0.
\end{equation}

If $E$ satiefies the condition {\rm(A)} of Theorem {\rm\ref{thm2}}, by Lemma \ref{l7}, we have $l(T_m) \ge(1-\frac{2\omega_1}{M})l(T_{m_{k-1}})$ for any $k\ge1 $ and $m_{k-1}\le m<m_k$. Notice that for any $\varepsilon>0$, there exists $N>0$, such that
$\frac{\log_{M} l(T_{m_k}) }{m_k}>-\frac{\varepsilon }{2} $ and $\frac{\log_{M}(1-\frac{2\omega_1}{M})^{-1} }{m_k}<\frac{\varepsilon }{2} $ for any $k\ge N$.
 Therefore if $m\ge m_{N}$, there is $h\ge N$ such that $m_h \le m < m_{h+1}$, then we have $$\frac{\log_M l(T_m)}{m}\ge\frac{\log_M l(T_{m_h})+\log_M(1-\frac{2\omega_1}{M})}{m_h}>-\varepsilon ,$$
which implies $\liminf\limits_{k \to\infty}   \frac{\log_M l(T_m ) }{m}=0$.
Since $l(T_m)\le 1$ for any $m\ge 0$,
$$\limsup\limits_{k \to\infty}  \frac{\log_M l(T_m) }{m}\le 0,$$ which implies that $$\lim_{m \to \infty} \frac{\log_M l(T_m) }{m}=0.$$

If $E$ satiefies the condition {\rm(B)} of Theorem {\rm\ref{thm2}}, by the similar proof (repalce $1-\frac{2\omega_1}{M}$ by $1-\frac{2(\omega_2+1)}{M}$), we have
$$\lim_{m \to \infty} \frac{\log_M l(T_m) }{m}=0.$$

(2) If $E$ satiefies the condition {\rm(A)} of Theorem {\rm\ref{thm2}}, for any $m_{k-1}\le m <m_k$, let $\kappa_m=\min\{\frac{\eta^*_{\sigma,l}}{|I|}:I\in\mathcal{T}_m, \sigma \in D_{k-1}, 1\le l\le n_k-1\}$. By the condition (A) of Theorem \ref{thm2} and (4) of Lemma \ref{l6}, we obtain $\beta_m\le 2\omega_1^2\kappa_m$. For any $0\le j\le m-1$ and $I\in T_j$, $\frac{T_{j+1}}{T_j}\le \frac{|I|-\kappa_j|I|}{|I|}=1-\kappa_j$, then  we have $l(T_m)\le \prod_{j=0}^{m-1}(1-\kappa_j)$. By the inequality $\log_M(1-x)\le -x$ for any $x\in[0,1)$, combining with (1), we have
$$0\ge \lim_{m\to \infty}-\frac{1}{m}\sum_{j=0}^{m-1}\kappa_j\ge \lim_{m\to \infty}\frac{1}{m}\sum_{j=0}^{m-1}\log(1-\kappa_j)\ge \lim_{m\to \infty}\frac{\log_{M}l(T_{m}) }{m}=0,$$
which implies that $$\lim_{m\to \infty}\frac{1}{m}\sum_{j=0}^{m-1}\kappa_j=0.$$
Then we have
$$0=2\omega^2\lim_{m\to \infty}\frac{1}{m}\sum_{j=0}^{m-1}\kappa_j\ge \lim_{m\to \infty}\frac{1}{m}\sum_{j=0}^{m-1}\beta_j\ge 0,$$
which implies that
$$\lim_{m\to \infty}\frac{1}{m}\sum_{j=0}^{m-1}\beta_j=0.$$

If $E$ satiefies the condition {\rm(B)} of Theorem {\rm\ref{thm2}}, by (1), we have
\begin{equation*}
\begin{aligned}
\lim_{k \to \infty} &\frac{1}{m_k} \log_{M}\prod_{j = 1}^{k}\frac{\delta^*_{j-1}-e^*_{j}-(L^*_j+R^*_j)}{\delta^*_{j-1}}\\
&= \lim_{k \to \infty}\frac{1}{m_k}\log_{M}\prod_{j = 1}^{k}\frac{n^*_j\delta^*_j}{\delta^*_{j-1}}\\
&= \lim_{k \to \infty}(\frac{1}{m_k}\log_{M}n^*_1n^*_2\cdots n^*_k\delta^*_k-\frac{1}{m_k}\log_{M}\delta^*_0)\\
&= \lim_{k \to \infty}(\frac{1}{m_k}\log_{M}l(S_{m_k}) -\frac{1}{m_k}\log_{M}\delta^*_0)\\
&=0.
\end{aligned}
\end{equation*}

By the inequality $\log_M(1-x)\le -x$ for any $x\in[0,1)$, we have $$0=\lim_{k \to \infty} \frac{1}{m_k} \log_{M}\prod_{j = 1}^{k}\frac{\delta^*_{j-1}-e^*_{j}-(L^*_j+R^*_j)}{\delta^*_{j-1}}\le\lim_{k \to \infty} -\frac{1}{m_k} \sum_{j = 1}^{k}\frac{e^*_{j}+(L^*_j+R^*_j)}{\delta^*_{j-1}}\le 0,$$
which implies that
$$\lim_{k \to \infty} \frac{1}{m_k} \sum_{j = 1}^{k}\frac{e^*_{j}+(L^*_j+R^*_j)}{\delta^*_{j-1}}=0.$$
Notice that for any $1\le j \le k$, $\bar\alpha^*_{j}=\bar\alpha_{j}+L^*_j+R^*_j\le e^*_{j}+L^*_j+R^*_j$, then
\begin{equation}\label{65}
\lim_{k \to \infty} \frac{1}{m_k} \sum_{j = 1}^{k}\frac{\bar{\alpha}^*_j}{\delta^*_{j-1}}=0.
\end{equation}

Next we estimate  $\beta_m$ for $m\ge0$. Let $k\in \mathbb{N} $ satisfying $m_{k-1}\le m < m_k$.
If $I\in T_{m_{k}-1}$, $I$ contains at least $2$ branches of $T_{m_k}$, notice that $\underline{\alpha}^*_k\ge L^*_k+R^*_k $,  therefore $\left | I \right | \ge 2\delta^*_k+(L^*_k+R^*_k)$.
If $I\in T_{m_{k}-2}$, $I$ contains at least $2M$ branches of $T_{m_k}$, therefore $\left | I \right | \ge 2M\delta^*_k+(2M-1)(L^*_k+R^*_k)$.
If $t\in \{1,2, \cdots, m_k-m_{k-1}\}$, then for any $I \in T_{m_k-t}$, $I$ contains at least $2M^{t-1}$ branches of $T_{m_k}$, therefore $\left | I \right | \ge 2M^{t-1}\delta^*_k+(2M^{t-1}-1)(L^*_k+R^*_k)$.
For any $L\in \mathcal{G}_m$, we have $\left | L \right | \le \bar{\alpha}^*_k,$
then for any $t\in \{1,2, \cdots, m_k-m_{k-1}\}$, we obtain
\begin{equation}
\begin{aligned}
\beta_{m_k-t}&\le \frac{\bar{\alpha}^*_k}{2M^{t-1}\delta^*_k+(2M^{t-1}-1)(L^*_k+R^*_k)}\\
&\le \frac{\bar{\alpha}^*_k}{2^{t}\delta^*_k+(2^{t}-1)(L^*_k+R^*_k)}\\
&\le \frac{\bar{\alpha}^*_k}{2^{t-1}(\delta^*_k+L^*_k+R^*_k)}.
\end{aligned}
\end{equation}
Therefore,
\begin{equation}\label{xxxx}
\begin{aligned}
\sum_{m = m_{k-1}}^{m_k-1} \beta_{m} = \sum_{t  = 1}^{i_k} \beta_{m_k-t}&\le\frac{\bar{\alpha}^*_k}{\delta^*_k+L^*_k+R^*_k}\sum_{t = 1}^{i_k}\frac{1}{2^{t-1}} \\
 &=\frac{\bar{\alpha}^*_k}{\delta^*_k+L^*_k+R^*_k}\sum_{t =0}^{i_k-1}\frac{1}{2^{t}}\\
&\le \frac{2\bar{\alpha}^*_k}{\delta^*_k+L^*_k+R^*_k}.
\end{aligned}
\end{equation}
Then we have
\begin{equation}\label{xs}
\frac{1}{m_k} \sum_{j=0}^{m_k-1} \beta_{j} \le \frac{2}{m_k}\sum_{j=1}^{k} \frac{\bar{\alpha}^*_j}{\delta^*_j+L^*_j+R^*_j}.
\end{equation}

For any $\varepsilon>0$, there exists $\delta>0$, such that
\begin{equation}\label{69}
0<\frac{1+\omega_2}{\log_M\frac{1}{(1+\omega_2)\delta}-1} <\frac{\varepsilon}{4}.
\end{equation}
If $j\ge1$ satisfying $\frac{\delta^*_j+L^*_j+R^*_j}{\delta^*_{j-1}} <\delta$.
Since $\bar{\alpha}_j\le \omega_2 \delta_j=\omega_2(\delta^*_j+L^*_j+R^*_j),$ we have
\begin{equation*}
\begin{aligned}
\delta^*_{j-1} & = e^*_{j}+L^*_j+R^*_j+n^*_j\delta^*_j\\
&=e_j+n^*_j(L^*_j+R^*_j)+n^*_j\delta^*_j\\
&\le (n^*_j-1)\bar{\alpha}_j+n^*_j(L^*_j+R^*_j+\delta^*_j)\\
&\le (n^*_j-1)\omega_2 (L^*_j+R^*_j+\delta^*_j)+n^*_j(L^*_j+R^*_j+\delta^*_j)\\
&\le n^*_j(1+\omega_2)(L^*_j+R^*_j+\delta^*_j),
\end{aligned}
\end{equation*}
which implies $(1+\omega_2)n^*_j\delta>(1+\omega_2)n^*_j\frac{L^*_j+R^*_j+\delta^*_j}{\delta^*_{j-1}} \ge 1.$
Notice that $i_j\ge \log_{M}n^*_j-1$, by (\ref{69}), we have
\begin{equation}\label{ss}
\frac{1+\omega_2}{i_j} <\frac{\varepsilon }{4}.
\end{equation}

By (\ref{65}), there exists $M_2>0$, such that for any $k\ge M_2$,
\begin{equation}\label{212}
\frac{1}{m_k} \sum_{j = 1}^{k}\frac{\bar{\alpha}^*_j}{\delta^*_{j-1}}<\frac{\varepsilon \delta}{4}.
\end{equation}
Notice that if $\bar{\alpha}_j\le \omega_2 \delta_j$, then
 \begin{align}\label{lwl}
\overline{\alpha } _{j}^{*}=\overline{\alpha}_j+L^*_j+R^*_j\le (\omega_2+1)\delta_j=(\omega_2+1)(\delta^*_{j}+L^*_j+R^*_j).
\end{align}
Therefore, if $k\ge M_2$, by (\ref{ss}),  (\ref{212}) and  (\ref{lwl}), we have
\begin{equation*}
\begin{aligned}
\frac{1}{m_k} \sum_{j = 1}^{k}\frac{\bar{\alpha}^*_j}{\delta^*_{j}+L^*_j+R^*_j}&\le \frac{1}{m_k}(\sum_{\substack{j=1\\\frac{\delta^*_j+L^*_j+R^*_j}{\delta^*_{j-1}}<\delta}}^{k}(1+\omega_2)+\sum_{\substack{j=1\\\frac{\delta^*_j+L^*_j+R^*_j}{\delta^*_{j-1}}\ge\delta}}^{k}\frac{\bar{\alpha}^*_j}{\delta^*_{j}+L^*_j+R^*_j} )\\
&\le \frac{1}{m_k}\sum_{j = 1}^{k}\frac{i_j\varepsilon }{4}+\frac{1}{m_k}\sum_{j = 1}^{k}\frac{\bar{\alpha}^*_j}{\delta^*_{j-1}}\cdot \frac{1}{\delta}\\
&<\frac{\varepsilon }{4}+\frac{\varepsilon }{4}=\frac{\varepsilon}{2},
\end{aligned}
\end{equation*}
which implies  $$\lim_{k \to \infty} \frac{1}{m_k} \sum_{j = 1}^{k}\frac{\bar{\alpha}^*_j}{\delta^*_{j}+L^*_j+R^*_j}=0.$$
Combing (\ref{xs}), we have
\begin{equation}\label{qwe}
\lim_{k \to \infty}\frac{1}{m_k} \sum_{j=0}^{m_k-1} \beta_{j}=0.
\end{equation}

For any  $k\ge 1$, if  $m_{k-1}<m<m_k$, then we have
\begin{equation}
\frac{1}{m} \sum_{j=0}^{m-1} \beta_{j}=\frac{1}{m} (\sum_{j=0}^{m_{k-1}-1} \beta_{j}+\sum_{j=m_{k-1}-1}^{m-1} \beta_{j})\le \frac{1}{m_{k-1}} \sum_{j=0}^{m_{k-1}-1} \beta_{j}+\frac{1}{m}\sum_{j=m_{k-1}-1}^{m-1}\beta_{j}.
\end{equation}
Notice that by (\ref{xxxx}) and (\ref{lwl}),
\begin{equation}\label{qas}
\sum_{j=m_{k-1}-1}^{m-1}\beta_{j}\le\frac{2\bar{\alpha}^*_k}{\delta^*_k+L^*_k+R^*_k}\le 2(1+\omega_2).
\end{equation}
Therefore, combing (\ref{qwe}) and (\ref{qas}), we obtain
\begin{equation*}
\lim_{m \to \infty} \frac{1}{m} \sum_{j=0}^{m-1} \beta_{j}=0.
\end{equation*}

(3) Fix $\varepsilon \in(0,\frac{1}{M^2+1} )$, such that $\log(1-(M^2+1)x)\ge -2(M^2+1)x$
for any $x\in[0,\varepsilon)$.
Then we have
\begin{equation*}
\begin{aligned}
0\ge \frac{1}{m} \sum_{\substack{j = 0\\\beta_j<\varepsilon }}^{m-1}\log(1-(M^2+1)\beta_j)\ge\frac{-2}{m} \sum_{\substack{j = 0\\\beta_j<\varepsilon }}^{m-1}(M^2+1)\beta_j&\ge
 -2(M^2+1)(\frac{1}{m}\sum_{j=0}^{m-1}\beta_j ).
\end{aligned}
\end{equation*}
By (2),
we have
\begin{equation*}
\lim_{m \to \infty} (\frac{1}{m} \sum_{\substack{j = 0\\\beta_j<\varepsilon }}^{m-1}\log(1-(M^2+1)\beta_j))=0,
\end{equation*}
which implies that
\begin{equation}\label{3a1}
\lim_{m \to \infty} ( \prod_{\substack{j = 0\\\beta_j<\varepsilon }}^{m-1}(1-(M^2+1)\beta_j))^{\frac{1}{m} }=1.
\end{equation}

If $E$ satiefies the condition {\rm(A)} of Theorem {\rm\ref{thm2}}, then each branch of $T_{m-1}$ contains at most $M^{2}$ branches of $T_{m}$ for any $m\ge1$, then we have $\frac{l(T_j) }{l(T_{j-1})}\le \min\{1,M^2\Lambda ^*(j)\}$ for any $1\le j \le m$. By (4) of Lemma \ref{l6}, we have $\Lambda^*(j)\le4\omega_1^2\Lambda_*(j)$ for any $j\ge 1$,
thus
\begin{equation}\label{3a2}
\prod_{j\in\Omega } (4M^2\omega_1^2\Lambda_*(j))\ge \prod_{j\in\Omega } (M^2\Lambda_*(j))\ge l(T_m)
\end{equation}
for any set $\Omega \subset \{1,2,\cdots,m\}$.

Let $H(m,\varepsilon)=\# (\{0\le j \le m-1:\beta_j<\varepsilon\}),$ combing (2) and Lemma \ref{l9}, we conclude that
\begin{equation}\label{3a3}
\lim_{m \to \infty} (1-\frac{H(m,\varepsilon )}{m}) =0.
\end{equation}

Notice that  for any $j\ge 0$, \begin{equation}\label{3a4}
\Theta_j=\min\{\frac{\sum_{i=1}^{N(I_j) }\left | I_{j,i} \right | }{\left | I_j \right | }: I_j\in T_j \}\ge \frac{\min_{I\in \mathcal{S}_{j+1} }\left | I \right | }{\max_{I\in \mathcal{S}_{j} }\left | I \right |}=\Lambda_*(j+1).
\end{equation}
Combing (\ref{3a2}), (\ref{3a4}) and Lemma \ref{l8}, we obtain

\begin{equation*}
\begin{aligned}
\lim_{m \to \infty} ( \prod_{j  = 0}^{m-1}\Theta_j)^{\frac{1}{m} } & =\lim_{m \to \infty} ( \prod_{\substack{j  = 0\\\beta_j<\varepsilon }}^{m-1}\Theta_j)^{\frac{1}{m} }( \prod_{\substack{j  = 0\\\beta_j\ge\varepsilon }}^{m-1}\Theta_j)^{\frac{1}{m} }\\
&\ge  \lim_{m \to \infty}(\prod_{\substack{j  = 0\\\beta_j<\varepsilon }}^{m-1}(1-(M^2+1)\beta_j))^{\frac{1}{m} }( \prod_{\substack{j  = 0\\\beta_j\ge\varepsilon }}^{m-1}\frac{1}{4M^2\omega_1^2} )^{\frac{1}{m} }l(T_m)^{\frac{1}{m} } \\
&=\lim_{m \to \infty}(\prod_{\substack{j  = 0\\\beta_j<\varepsilon }}^{m-1}(1-(M^2+1)\beta_j))^{\frac{1}{m} }(\frac{1}{4M^2\omega_1^2} )^{1-\frac{H(m,\varepsilon )}{m} }l(T_m)^{\frac{1}{m} }.
\end{aligned}
\end{equation*}

By (\ref{3a1}), we have $$\lim_{m \to \infty}(\prod_{\substack{j  = 0\\\beta_j<\varepsilon }}^{m-1}(1-(M^2+1)\beta_j))^{\frac{1}{m} }=1.$$
And by (\ref{3a3}) and (1), $$\lim_{m \to \infty}(\frac{1}{4M^2\omega_1^2} )^{1-\frac{H(m,\varepsilon )}{m} }l(T_m)^{\frac{1}{m} }=1.$$
Then $\lim_{m \to \infty} ( \prod_{j  = 0}^{m-1}\Theta_j)^{\frac{1}{m} }=1$, which implies
\begin{equation*}
\lim_{m \to \infty}\frac{1}{m} \sum_{j=0}^{m-1}\log \Theta_j=0 .
\end{equation*}

If $E$ satiefies the condition {\rm(B)} of Theorem {\rm\ref{thm2}}, by the similar proof(repalce $\omega_1$ by $\omega_2+1$), we also have
\begin{equation*}
\lim_{m \to \infty}\frac{1}{m} \sum_{j=0}^{m-1}\log \Theta_j=0 .
\end{equation*}

(4) If $E$ satiefies the condition {\rm(A)} of Theorem {\rm\ref{thm2}}, by (4) of Lemma \ref{l6}, for any $j\ge 1$, $J\in \mathcal{T}_j$ and $J^{\prime}\in \mathcal{T}_{j-1}$, we have
\begin{equation*}
\chi_j\le\frac{\max_{\hat{\hat{J}}\in T_j}\left | \hat{\hat{J}} \right | }{\min_{\hat{J}\in T_{j-1}}\left | \hat{J} \right |}\le\frac{2\omega_1\left | J \right | }{\frac{1}{2\omega_1} \left |J^{\prime} \right |}\le 4\omega_1^2\frac{\left | J \right |}{\left |J^{\prime} \right |}.
\end{equation*}

Taking $J^0\in T_j$ which satisfies $\chi_j=\frac{\left | J^0 \right |}{\left |X_a(J^0) \right |} $.
Since $X_a(J^0)$  contains at least $2$ branches of $T_j$, we have
\begin{equation*}
1>\chi_j+\frac{\left | J^* \right |}{\left |X_a(J^0) \right |}\ge \chi_j+\frac{\chi_j}{4\omega_1^2}=(\frac{4\omega_1^2+1}{4\omega_1^2} )\chi_j,
\end{equation*}
where $J^*\in \mathcal{T}_j$, $J^*\subset X_a(J^0)$ and $J^*\neq J$. Let $\alpha \in (\frac{4\omega_1^2}{4\omega_1^2+1},1)$, then
\begin{equation*}
\lim_{m\to \infty} \frac{\#S(m,\alpha )}{m} =\lim_{m\to \infty}\frac{\#\{1\le i\le m:\chi_i<\alpha\}}{m} =1.
\end{equation*}

If $E$ satiefies the condition {\rm(B)} of Theorem {\rm\ref{thm2}},
 by the similar proof(repalce $\omega_1$ by $\omega_2+1$), we also have
\begin{equation*}
\lim_{m\to \infty} \frac{\#S(m,\alpha )}{m} =\lim_{m\to \infty} \frac{\#\{1\le i\le m:\chi_i<\alpha\}}{m} =1.
\end{equation*}

\end{proof}

\subsection{The measure supported on $f(E)$}
Let $E=E(I_{0},\left\{n_{k}\right\},\left\{c_{k}\right\})$ be a homogeneous Moran set which satisfies the conditions of Theorem \ref{thm2}, $f$ be a \text{\rm1}-dimensional quasisymmetric mapping, and $\{T_{m}\}_{m\ge0}$ be the sequences in Lemma \ref{l6}. We are going to define a probability Borel measure on $f(E)$ to estimate the lower bound of the Hausdorff dimension of $f(E)$ by Lemma \ref{l1}.

For any $m\ge0$ and any  branch of $T_m$, denoted by $I_m$, let $J_{m}=f(I_m)$. Notice that the image sets of all branches of $T_m$ under $f$ constitute $f(T_m)$, for convenience,  we also call $J_m$ a branch of $f(T_m)$. Let $J_{m,1}\cdots,J_{m,N(J_{m})}$ be all branches of $f(T_{m+1})$ contained  in $J_{m}$  locating from left to right, where $N(J_{m})$ is the number of the branches of $f(T_{m+1})$ contained in $J_{m}$, then $N(J_{m})\le M^{2}$.

For any $d\in(0,1)$, $m\ge0$ and $1\le i\le N(J_{m})$, define a probability Borel measure $\mu_{d}$ on $f(E)$ satisfying $\mu _d(f(T_0))=1$ and
\begin{equation}
\mu_d(J_{m,i})=\frac{\left | J_{m,i} \right |^d }{\sum_{j=1}^{N(J_{m})}\left | J_{m,j} \right |^d }\mu_d(J_{m}).
\end{equation}
\bigskip

Next, for any $d\in(0,1)$ and $k\ge 1$, we estimate  $\mu_d(U)$ for any branch $U$ of $f(T_k)$.

\begin{proposition}\label{p1} For any $d\in(0,1), k\ge1$, let $U=J_k$ be a branch of $f(T_k)$, then there exists a contant $C_1>0$, such that $\mu _d(U)\le C_1|U|^{d}$.
\end{proposition}

\begin{proof}
If $E$ satiefies the condition {\rm(A)} or the condition {\rm(B)} of Theorem {\rm\ref{thm2}}, for any $d\in(0,1)$, $k\ge 1$,
if $U=J_k$ is a branch of $f(T_k)$, then for any $0\le j\le k-1$, suppose $J_j$ is a branch of $f(T_j)$ which contains $U$, then
$U=J_k\subset J_{k-1}\subset \cdots \subset J_1\subset J_0=f(T_0).$ Without loss of generality, we  assume that $J_0=f(T_0)=[0,1]$.
By the definition of $\mu_d$, we have

\begin{equation*}
\frac{\mu_d(J_k)}{\left | J_k \right |^d }=\prod_{j=0}^{k-1}\frac{\left | J_j \right |^d}{\sum_{i=1}^{N(J_j)}\left | J_{j,i} \right |^d}.
\end{equation*}
Then if we prove
\begin{equation*}
\liminf\limits_{k \to\infty} (\prod_{j=0}^{k-1}\frac{\sum_{i=1}^{N(J_j)}\left | J_{j,i} \right |^d}{\left | J_j \right |^d})^{\frac{1}{k}}>1,
\end{equation*}
we can finish the proof of  Proposition \ref{p1}.

We will estimate $\frac{\sum_{i=1}^{N(J_j)}\left | J_{j,i} \right |^d}{\left | J_j \right |^d}$ for any $0\le j\le k-1.$ Suppose $J_{j,1},\cdots, J_{j,N(J_j)}$ are all branches of $f(T_{j+1})$ contained in $J_j$ locating from left to right, and  $I_j=f^{-1}(J_j)$ is a branch of $T_j$. For any $1\le l\le N(J_j)-1$,  let
$$L_{j,0}=\big[\min(J_j),\min(J_{j,1})\big),~L_{j,N(J_j)}=\big(\max(J_{j,N(J_j)}),\max(J_j)\big],$$
$$L_{j,l}=\big(\max(J_{j,l}),\min(J_{j,l+1})\big).$$
Then $J_j=\big(\bigcup_{i=1}^{N(J_j)}J_{j,i} \big)\bigcup \big(\bigcup_{l=0}^{N(J_j)}L_{j,l} \big)$. Notice that it may exists $0\le l\le N(J_j)$, such that $L_{j,l}=\emptyset$.
Let $G_{j,l}=f^{-1}(L_{j,l})\subset I_j-T_{j+1}$ for any $0\le l \le N(J_j)$.

We decompose the estimation formula such as
\begin{equation}\label{p11}
\frac{\sum_{i=1}^{N(J_j)}\left | J_{j,i} \right |^d }{\left | J_j \right |^d }= \frac{\sum_{i=1}^{N(J_j)}\left | J_{j,i} \right |^d }{(\sum_{i=1}^{N(J_j)}\left | J_{j,i} \right |)^d }\frac{{(\sum_{i=1}^{N(J_j)}\left | J_{j,i} \right |)^d }}{\left | J_j \right |^d}.
\end{equation}

Let $\alpha \in (0,1)$ be the constant in (4) of Lemma \ref{l10}, $p \in (0,1]$ be the constant in Lemma \ref{l2}, $\varepsilon>0$ be a sufficiently small constant satisfying
\begin{enumerate}
\item[\textup{(1)}]$0<\varepsilon<\frac{1-\alpha}{M^2+1}$;
\item[\textup{(2)}]$(1-4(M^2+1)x^p)\ge (1-x^p)^{4(M^2+1)}$ for any $x\in[0,\varepsilon)$;
\item[\textup{(3)}]$\log(1-x^p)\ge -2x^p$  for any $x\in[0,\varepsilon)$.
\end{enumerate}

Without loss of generality, suppose $\left | J_{j,1} \right |=\max_{1\le i \le N(J_j)}\{\left | J_{j,i} \right |\}  $, $y_l=\frac{\left | J_{j,l} \right | }{\left | J_{j,1} \right | } $ for any $1\le l\le N(J_j)$, then we have
\begin{equation}\label{333}
\begin{aligned}
\frac{\sum_{i  = 1}^{N(J_j)}\left | J_{j,i} \right |^d }{(\sum_{i  = 1}^{N(J_j)}\left | J_{j,i} \right |)^d }&=\frac{y_1^d+y_2^d+\cdots+y_{N(J_j)}^d}{(y_1+y_2+\cdots+y_{N(J_j)})^d}\\
&=  \frac{1+y_2^d+\cdots+y_{N(J_j)}^d}{(1+y_2+\cdots+y_{N(J_j)})^d}\\
&\ge (1+y_2+\cdots+y_{N(J_j)})^{1-d}>1.
\end{aligned}
\end{equation}
Therefore,
\begin{equation}\label{p12}
\frac{\sum_{i=1}^{N(J_j)}\left | J_{j,i} \right |^d }{\left | J_j \right |^d }= \frac{\sum_{i=1}^{N(J_j)}\left | J_{j,i} \right |^d }{(\sum_{i=1}^{N(J_j)}\left | J_{j,i} \right |)^d }\frac{{(\sum_{i=1}^{N(J_j)}\left | J_{j,i} \right |)^d }}{\left | J_j \right |^d}\ge \frac{{(\sum_{i=1}^{N(J_j)}\left | J_{j,i} \right |)^d }}{\left | J_j \right |^d}.
\end{equation}

\textbf{(a)} If $\beta_j<\varepsilon$, then $\frac{\left | G_{j,l} \right | }{\left | I_j \right | } \le\beta_j$  for any $0\le l \le N(J_j)$. By Lemma \ref{l2}, $\frac{\left | L_{j,l} \right | }{\left | J_j \right | }\le4(\frac{\left | G_{j,l} \right | }{\left | I_j \right | })^p \le4(\beta_j)^p $, then we have

\begin{equation}\label{619}
(\frac{{\sum_{i=1}^{N(J_j)}\left | J_{j,i} \right | }}{\left | J_j \right |})^d=(\frac{\left | J_{j} \right |-{\sum_{i=0}^{N(J_j)}\left | L_{j,i} \right | }}{\left | J_j \right |})^d\ge(1-4(M^2+1)\beta^p_j)^d\ge(1-\beta^p_j)^{4(M^2+1)d}.
\end{equation}
Moreover, if $\beta_j<\varepsilon$ and $\chi_{j+1}<\alpha$, by Lemma \ref{l2} and the Jensen inequality,
we obtain
\begin{equation}\label{620}
\frac{{\sum_{l=2}^{N(J_j)}\left | J_{j,l} \right | }}{\left | J_j \right |}\ge\lambda \frac{{\sum_{l=2}^{N(J_j)}\left | I_{j,l} \right |^q }}{\left | I_j \right |^q}\ge(M^2-1)^{1-q}\lambda(\frac{{\sum_{l=2}^{N(J_j)}\left | I_{j,l} \right | }}{\left | I_j \right |})^q.
\end{equation}

Since $\frac{\left | I_{j,l} \right | }{\left | I_j \right | }\le\chi_{j+1}<\alpha$ for any $1\le l \le N(J_j)$, and $\frac{\left | G_{j,l} \right | }{\left | I_j \right | }\le \beta_j<\varepsilon $ for any $0\le l \le N(J_j)\le M^2$, we conclude that
\begin{equation}\label{621}
\frac{\sum_{l=2}^{N(J_j)}\left | I_{j,l} \right | }{\left | I_j \right |}=\frac{\left | I_{j} \right |-\left | I_{j,1} \right |-\sum_{l=0}^{N(J_j)}\left | G_{j,l} \right |  }{\left | I_j \right |}\ge1-\alpha-(M^2+1)\varepsilon .
\end{equation}
Combining (\ref{620}) and (\ref{621}), we obtain

\begin{equation}\label{p13}
\frac{{\sum_{l=2}^{N(J_j)}\left | J_{j,l} \right | }}{\left | J_j \right |}\ge (M^2-1)^{1-q}\lambda(1-\alpha-(M^2+1)\varepsilon )^q.
\end{equation}

By Lemma \ref{l2}, for any $1\le l \le N(J_j)$, we have
\begin{equation*}
\frac{\left | J_{j,l} \right | }{\left | J_j \right |}=\frac{\left | f(I_{j,l}) \right | }{\left | f(I_j) \right |}\le4\frac{\left | I_{j,l} \right |^p }{\left | I_j \right |^p}\le 4\alpha^p.
\end{equation*}
Hence,
\begin{equation}\label{3333}
\begin{aligned}
y_2+y_3+\cdots+y_{N(J_j)}=\frac{\left | J_j \right |}{\left | J_{j,1} \right | }\frac{{\sum_{l=2}^{N(J_j)}\left | J_{j,l} \right | }}{\left | J_j \right |}&\ge \frac{\left | J_j \right |}{\left | J_{j,1} \right | }\frac{\lambda(1-\alpha-(M^2+1)\varepsilon )^q}{(M^2-1)^{q-1}}\\&\ge \frac{\lambda(1-\alpha-(M^2+1)\varepsilon )^q}{4\alpha^p(M^2-1)^{q-1}}.
\end{aligned}
\end{equation}

By (\ref{p12}) and (\ref{619}), if $\beta_j<\varepsilon$, we have
\begin{equation}\label{637}
\frac{\sum_{i=1}^{N(J_j)}\left | J_{j,i} \right |^d }{\left | J_j \right |^d }\ge (1-\beta^p_j)^{4(M^2+1)d}.
\end{equation}
If $\beta_j<\varepsilon$ and $\chi_{j+1}<\alpha$, combing (\ref{p11}),  (\ref{333}), (\ref{619}) and (\ref{3333}), we have
\begin{equation}\label{638}
\frac{\sum_{i=1}^{N(J_j)}\left | J_{j,i} \right |^d }{\left | J_j \right |^d }\ge \eta(1-\beta^p_j)^{4(M^2+1)d},
\end{equation}
where $\eta=(1+\frac{\lambda(1-\alpha-(M^2+1)\varepsilon )^q}{4\alpha^p(M^2-1)^{q-1}})^{1-d}>1.$

On other hand, for $\beta_j<\varepsilon$, we have
\begin{equation*}
\begin{aligned}
0\ge \frac{1}{k} \sum_{\substack{j = 0\\\beta_j<\varepsilon }}^{k-1}\log(1-\beta^p_j)\ge\frac{-2}{k} \sum_{\substack{j = 0\\\beta_j<\varepsilon }}^{k-1}\beta^p_j&\ge\frac{-2}{k}\sum_{j = 0}^{k-1}\beta^p_j\\
 &\ge-2(\frac{1}{k}\sum_{j=0}^{k-1}\beta_j )^p.
\end{aligned}
\end{equation*}
for any $k\ge 1$. By (2) of Lemma \ref{l10}, we have $\lim_{k \to \infty} \frac{1}{k} \sum_{\substack{j = 0\\\beta_j<\varepsilon }}^{k-1}\log(1-\beta^p_j)=0,$ which implies that
\begin{equation}\label{639}
\lim_{k \to \infty} [\prod_{\substack{j = 0\\\beta_j<\varepsilon }}^{k-1}(1-\beta^p_j)]^{\frac{1}{k}}=1.
\end{equation}

\textbf{(b)} If $\beta_j\ge \varepsilon$, by  Lemma \ref{l2}  and the Jensen inequality, we have

\begin{equation*}
\frac{\sum_{l=1}^{N(J_j)}\left | J_{j,l} \right | }{\left | J_j \right |}\ge \lambda\frac{\sum_{l=1}^{N(J_j)}\left | I_{j,l} \right |^q }{\left | I_j \right |^q}\ge \frac{\lambda}{M^{2(q-1)}}(\frac{\sum_{l=1}^{N(J_j)}\left | I_{j,l} \right |}{\left | I_j \right |})^q \ge\frac{\lambda}{M^{2(q-1)}}\Theta^q_j.
\end{equation*}
By (\ref{p12}),
\begin{equation}\label{640}
\frac{\sum_{l=1}^{N(J_j)}\left | J_{j,l} \right |^d }{\left | J_j \right |^d}\ge \frac{(\sum_{l=1}^{N(J_j)}\left | J_{j,l} \right |)^d }{\left | J_j \right |^d}\ge (\frac{\lambda}{M^{2(q-1)}}\Theta^q_j)^d.
\end{equation}

For any $k\ge1$, let $P(k)=\#(\{0\le j \le k-1: \beta_j<\varepsilon\})$, $R(k)=\#(\{1\le j \le k: \chi_j<\alpha\})$ and $PR(k)=\#(\{1\le j \le k: \beta_{j-1}<\varepsilon, \chi_j<\alpha\})$$(\#$ denotes the cardinality$)$.
Notice that by (2) of Lemma \ref{l10},
\begin{equation*}
\lim_{k\to \infty} \frac{1}{k} \sum_{j=0}^{k-1}\beta_j=0,
\end{equation*}
by Lemma \ref{l9}, we have
\begin{equation}\label{pk}
\lim_{k\to \infty} \frac{P(k)}{k}=1.
\end{equation}
By (4) of Lemma \ref{l10}, we have
\begin{equation*}
\lim_{k\to \infty} \frac{R(k)}{k}=1,
\end{equation*}
then
\begin{equation}\label{prk}
\lim_{k\to \infty} \frac{PR(k)}{k}=1.
\end{equation}

Combing (\ref{637}), (\ref{638}) and (\ref{640}), we obtain

\begin{equation*}
\begin{aligned}
\prod_{j=0}^{k-1} \frac{\sum_{i=1}^{N(J_j)}\left | J_{j,i} \right |^d }{\left | J_j \right |^d }&= \prod_{\substack{j = 0\\\beta_j<\varepsilon,\chi_{j+1}<\alpha}}^{k-1} \frac{\sum_{i=1}^{N(J_j)}\left | J_{j,i} \right |^d }{\left | J_j \right |^d }\prod_{\substack{j = 0\\\beta_j<\varepsilon,\chi_{j+1}\ge\alpha}}^{k-1} \frac{\sum_{i=1}^{N(J_j)}\left | J_{j,i} \right |^d }{\left | J_j \right |^d }\prod_{\substack{j = 0\\\beta_j\ge\varepsilon}}^{k-1} \frac{\sum_{i=1}^{N(J_j)}\left | J_{j,i} \right |^d }{\left | J_j \right |^d }\\
&\ge \eta^{PR(k)}\prod_{\substack{j = 0\\\beta_j<\varepsilon}}^{k-1}(1-\beta^p_j)^{4(M^2+1)d}\prod_{\substack{j = 0\\\beta_j\ge\varepsilon}}^{k-1}(\frac{\lambda}{M^{2(q-1)}}\Theta^q_j)^d\\
&\ge\eta^{PR(k)}\prod_{\substack{j = 0\\\beta_j<\varepsilon}}^{k-1}(1-\beta^p_j)^{4(M^2+1)d}(\prod_{j=0}^{k-1}\Theta_j )^{qd}\prod_{\substack{j = 0\\\beta_j\ge\varepsilon}}^{k-1}(\frac{\lambda}{M^{2(q-1)}})^d\\
&= \eta^{PR(k)}\prod_{\substack{j = 0\\\beta_j<\varepsilon}}^{k-1}(1-\beta^p_j)^{4(M^2+1)d}(\prod_{j=0}^{k-1}\Theta_j )^{qd}(\frac{\lambda}{M^{2(q-1)}})^{d(k-P(k))}.\\
\end{aligned}
\end{equation*}
Combing (\ref{639}), (\ref{pk}),   (\ref{prk}) and (3) of Lemma \ref{l10}, we have
\begin{equation*}
\liminf\limits_{k \to\infty} (\prod_{j=0}^{k-1} \frac{\sum_{i=1}^{N(J_j)}\left | J_{j,i} \right |^d }{\left | J_j \right |^d })^{\frac{1}{k} }\ge \liminf\limits_{k \to\infty} \eta^{\frac{PR(k)}{k}}(\prod_{j=0}^{k-1}\Theta_j )^{\frac{qd}{k}}(\frac{\lambda}{M^{2(q-1)}})^{\frac{d(k-P(k))}{k}}=\eta>1.
\end{equation*}

Thus, there exists a constant $C_1>0$ such that
\begin{equation*}
\frac{\mu_d(J_k)}{\left | J_k \right |^d }=\prod_{j=0}^{k-1}\frac{\left | J_j \right |^d}{\sum_{i=1}^{N(J_j)}\left | J_{j,i} \right |^d}\le C_1.
\end{equation*}

\end{proof}

\subsection{The proof of Theorem \ref{thm2} }
We begin to finish the proof Theorem \ref{thm2}. Let $E=E(I_{0},\left\{n_{k}\right\},\left\{c_{k}\right\})$ be a homogeneous Moran set which satisfies the conditions of Theorem \ref{thm2}, $f$ be a \text{\rm1}-dimensional quasisymmetric mapping, for any $x\in f(E)$, define $\delta=\sup\{r:|f^{-1}(B(x,r))|<\delta^{*}_{0}\}$.
Since $f$ is a homeomorphism, $F_{x}(r)=|f^{-1}(B(x,r)|$ is a monotonically increasing function with $\lim_{r\to 0}F_{x}(r)=0$.

(i) If $E$ satiefies the condition {\rm(A)} of Theorem {\rm\ref{thm2}}, then for any $0<r<\delta$, there exists a   positive integer $m$ satisfying
$$\min_{I\in \mathcal{T}_m }\left | I \right |\le \left | f^{-1}\big(B(x,r)\big) \right |< \min_{I\in \mathcal{T}_{m-1} }\left | I \right |. $$
Which implies that  the number of the branches of $\mathcal{T}_{m-1}$ intersecting  $f^{-1}(B(x,r))$ is at most $2$, then $f^{-1}(B(x,r))$ intersects at most $2M^2$ branches of $\mathcal{T}_{m}$, therefore
$B(x,r)$ intersects at most $2M^2$ branches of $f(T_{m})$. The branches of $f(T_{m})$ which intersect $B(x,r)$ is denoted by $U_1,U_2, \cdots, U_l(1\le l\le2M^2)$, then
$${B(x,r) }\cap f(E)\subset U_1\cup U_2 \cup \cdots \cup U_l.$$

By proposition $1$, we have
\begin{equation}\label{627}
\mu_d\big(B(x,r)\big)=\mu_d\big(B(x,r)\cap f(E)\big)\le \sum_{j=1}^{l}\mu_d(U_j)\le C_1\sum_{j=1}^{l}\left | U_j \right |^d.
\end{equation}
Notice that
$$\min_{I\in \mathcal{T}_m }\left | I \right |\le \left | f^{-1}\big(B(x,r)\big) \right |,\quad \max_{I\in \mathcal{T}_m }\left | I \right |\le2\omega_1 \min_{I\in \mathcal{T}_m }\left | I \right |,$$
then for any $1\le j \le l$, we have
$$\left | f^{-1}(U_j) \right | \le \max_{I\in \mathcal{I}_m }\left |I \right | \le 2\omega_1 \min_{I\in \mathcal{I}_m }\left | I \right |\le 2\omega_1 \left | f^{-1}\big(B(x,r)\big) \right |.$$
Since $B(x,r)\cap U_j\ne \emptyset$, we obtain
$$f^{-1}(U_j)\subset 6\omega_1 f^{-1}\big(B(x,r)\big),$$
where for any interval $I$ and $\rho>0$, $\rho I$ is the interval which has the same center with $I$ and length of it is $\rho \left|I\right|$.

By  Lemma $2$, since $f$ is a homeomorphism, we have
\begin{equation}\label{628}
\left | U_j \right |\le \left | f\Big(6\omega_1 f^{-1}\big(B(x,r)\big)\Big) \right |\le K_{6\omega_1}\left | B(x,r) \right |\le 2K_{6\omega_1}r,
\end{equation}
then by (\ref{627}), (\ref{628}) and $1\le l \le 2M^2$, we obtain
\begin{equation}
\begin{aligned}
\mu_d\big(B(x,r)\big)&\le C_1\sum_{j = 1}^{l}\left | U_j \right |^d\\
&\le C_1 \cdot  2M^2(2K_{6\omega_1}r)^d\\
&\le 4K^d_{6\omega_1}M^2C_1r^d\\
&\triangleq C_2r^d,\nonumber
\end{aligned}
\end{equation}
therefore
$$\limsup\limits_{r \to 0}\frac{\mu_d\big(B(x,r)\big)}{r^d}\le C_2.$$

Since $x\in f(E)$ is arbitrary, we have $\dim_{H}f(E)\ge d$ by  (2) of Lemma \ref{l1}. Since $d\in(0,1)$ is arbitrary, we obtain that $\dim_H f(E)\ge 1$. It is obvious that $\dim_H f(E)\le 1$, then we have $\dim_{H}f(E)=1$.

(ii) If $E$ satiefies the condition {\rm(B)} of Theorem {\rm\ref{thm2}}, By the similar proof of (i)(repalce $\omega_1$ by $\omega_2+1$), we obtain that there is a constant $C_3>0$ satisfying
$$\limsup\limits_{r \to 0}\frac{\mu_z\big(B(x,r)\big)}{r^z}\le C_3.$$

Since $x\in f(E)$ is arbitrary, we have $\dim_{H}f(E)\ge d$ by  (2) of Lemma \ref{l1}. Since $d\in(0,1)$ is arbitrary, we obtain that $\dim_H f(E)\ge 1$. It is obvious that $\dim_H f(E)\le 1$, then we have $\dim_{H}f(E)=1$.

We finish the proof of Theorem 2.

\bigskip

\textbf{Acknowledgement} The authors thank the reviewers for their helpful comments and suggestions.

\bigskip



\bigskip
\begin{thebibliography}{99}
\bibitem{fdj97} Feng D J, Wen Z Y, Wu J. Some dimensional results for homogeneous Moran sets. Science in China Series A: Mathematics, 1997, 40(5): 475-482.

\bibitem{wzywuj05} Wen Z Y, Wu J. Hausdorff dimension of homogeneous perfect sets. Acta Mathematica Hungarica, 2005, 107(1): 35-44.

\bibitem{Ahl06} Ahlfors V. \newblock Lectures on quasiconformal mappings. 2nd ed, Maryland: Vol. 38 of Unversity Lecture Series, American Mathematical Society, 2006.



\bibitem{klv06} Kovalev V. Conformal dimension does not assume values between zero and one. Duke Mathematical Journal, 2006, 134(1): 1-13.

\bibitem{bcj99} Bishop J. Quasiconformal mappings which increase dimension. Annales Academiae Scientiarum Fennicae-mathematica, 1999, 24(2): 397-407.
\bibitem{gfw} Gehring W, Vaisala J. \newblock Hausdorff dimension and quasiconformal mappings. Journal of the London Mathematical Society, 1973, 6: 504-512.

\bibitem{gfw2} Gehring  W. \newblock The $L^p$-integrability of the partial derivatives of a quasiconformal mapping. Bulletin of the American Mathematical Society, 1973, 79: 465-466.
\bibitem{tp89} Tukia P. \newblock Hausdorff dimension and quasisymmetric mappings. Mathematica Scandinavica, 1989, 65(1): 152-160.

\bibitem{ssg98} Staples G, Ward A. \newblock Quasisymmetrically thick sets. Annales Academiae Scientiarum Fennicae-mathematica, 1998, 23: 151-168.

\bibitem{hh06} Hakobyan H. \newblock Cantor sets that are minimal for quasisymmetric mappings. Journal of Contemporary Mathematical Analysis, 2006, 41(2): 13-21.

\bibitem{hmd08} Hu M D, Wen S Y. \newblock Quasisymmetrically minimal uniform Cantor sets. Topology and its Applications, 2008, 155: 515-521.

\bibitem{ww14} Wang W, Wen S Y. \newblock Quasisymmetric minimality of Cantor sets. Topology and its Applications, 2014, 178: 300-314.

\bibitem{dyx11} Dai Y X, Wen Z Y, Xi L F, et al. \newblock Quasisymmetrically minimal Moran sets and Hausdorff dimension. Annales Academiae Scientiarum Fennicae-mathematica, 2011, 36: 139-151.

\bibitem{yjj18} Yang J J, Wu M, Li Y Z. \newblock On quasisymmetric minimality of homogeneous perfect sets. Fractals, 2018, 26(1): 1850010.

\bibitem{XYQ} Xiao Y Q, Zhang Z Q. On the quasisymmetric minimality of homogeneous perfect sets(in Chinese). Acta Mathematica Sinica, Chinese Series, 2019, 62(4): 573-590.


\bibitem{hua00} Hua S, Rao H, Wen Z Y,  et al. \newblock On the structures and dimensions of Moran sets. Science in China Series A: Mathematics, 2000, 43(8): 836-852.


\bibitem{lss24} Liu S S, Li Y Z, Zong W Q, et al. \newblock Hausdorff dimension and upper box dimension of a class of homogeneous Moran sets. Dynamical Systems, 2024, 39(3): 449-460.
\bibitem{LW10} Lou M L, Wu M.
\newblock The pointwise dimensions of Moran measure. Science China Mathematics, 2010, 53(5): 1283--1292.
\bibitem{LW11}Li J J, Wu M. \newblock Pointwise dimensions of general Moran measures with open set condition.   Science China Mathematics, 2011, 54(4): 699-710.
\bibitem{CWW17} Chen H P, Wu M, Wei C. \newblock Lower dimensions of some fractal sets. Journal of Mathematical Analysis and Applications, 2017, 455(2): 1022-1036.


\bibitem{LFY21} Li Y Z, Fu X H, Yang  J J. \newblock  Quasisymmetrically minimal Moran sets on packing dimension. Fractals, 2021, 29(2): 2150043.
\bibitem{DDW23} Dai Y X, Dong J M, Wei C. \newblock   The continuity of dimensions and quasisymmetrical equivalence of parameterized homogeneous Moran sets. Journal of Mathematical Analysis and Applications,  2023, 518(2):  126783.
\bibitem{LLY25}Liu S S, Li Y Z, Yang J J. \newblock  Quasisymmetric minimality on packing dimension for homogeneous perfect sets.  Axioms, 2025, 14(10): 14100751.


\bibitem{wen00} Wen Z Y. \newblock Fractal geometry-mathematical foundation. Shang Hai: Shanghai Science and Technology Education Press, 2000.

\bibitem{fkj97} Falconer J. \newblock Techniques in Fractal Geometry. Chichester: John Wiley and Sons, 1997.

\bibitem{wjm93} Wu J M. \newblock Null sets for doubling and dyadic doubling measures. Annales Academiae Scientiarum Fennicae Series A. I. Mathematica, 1993, 18(1): 77-91.



\end{thebibliography}
\end{document}